\documentclass[12pt,a4paper]{amsart}
\allowdisplaybreaks % for proof reading

\usepackage{color,ulem}
\usepackage{amsmath}
\usepackage{amssymb}
\usepackage{amsthm}
\usepackage{amsfonts}
\usepackage{latexsym}
\usepackage{hyperref}
\usepackage{tikz}
\usepackage{enumerate}
\usetikzlibrary{matrix,arrows}

\usepackage{mathtools}
\usepackage{mhchem}
\usepackage{gensymb}

\newtheorem{theorem}{Theorem}[section]
\newtheorem{corollary}[theorem]{Corollary}
\newtheorem{lemma}[theorem]{Lemma}
\newtheorem{proposition}[theorem]{Proposition}
\theoremstyle{definition}
\newtheorem*{definition}{Definition}
\newtheorem{remark}[theorem]{Remark}
\numberwithin{equation}{section}
\newtheorem*{theorem*}{Theorem}

\newcommand\pref[1]{~(\ref{#1})}

\def \bC {\mathbb C}
\def \bD {\mathbb D}

\def \bN {\mathbb N}

\def \bR {\mathbb R}

\def \bR {\mathbb R}
\def \bR {\mathbb R}

\def \bZ {\mathbb Z}

\def \cD {\mathcal D}

\def \cG {\mathcal G}

\def \cP {\mathcal P}

\def \cS {\mathcal S}

\def \fg {\mathfrak g}

\def \fk {\mathfrak k}

\def \fm {\mathfrak m}

\def \fs {\mathfrak s}
\def \ft {\mathfrak t}
\def \fu {\mathfrak u}

\def \fU {\mathfrak U}

\def \RE {\text{\rm Re}\,}
\def \IM {\text{\rm Im}\,}
\def \al {\alpha}
\def \la {\lambda}
\def \ph {\varphi}

\def \eps {\varepsilon}
\def \lan {\langle}
\def \ran {\rangle}
\def \de {\partial}

\def \inv{^{-1}}

\def \deg {\text{\rm deg\,}}
\def \dim {\text{\rm dim\,}}
\def \spaz {\text{\rm span\,}}

\def \tr {\text{\rm tr\,}}
\def \DD {\mathbf D}

\def\be{\begin{equation}}
\def\ee{\end{equation}}
\def\bes{\begin{equation*}}
\def\ees{\end{equation*}}
\def\bea{\begin{equation}\begin{aligned}}
\def\eea{\end{aligned}\end{equation}}
\def\beas{\begin{equation*}\begin{aligned}}
\def\eeas{\end{aligned}\end{equation*}}

\newcommand\grpU{{U_2}}
\newcommand\grpSU{{SU_2}}
\newcommand\base{\mathbf{o}} %%% punto base sulla sfera

\newcommand\cGn{\cG_n}
\newcommand\cGmn{\cG_{\tau_{m,n}}}

\newcommand\baseuno {X_1}
\newcommand\basedue{X_2}
\newcommand\basetre{X_3}
\newcommand\basequattro{X_4}
\newcommand\Cas{\Omega}

\newcommand\diagonali{\mathcal{B}}
\newcommand\spaziorp{{V}}
\newcommand\spaziW{\mathcal{W}}

\newcommand\Smatr{\cS\big(\bC^2,{\rm End}(\spaziorp_n)\big)^K}
\newcommand\SmatrW{\cS(\bC^2,\spaziW_n^\ell)}
\newcommand\Lmatr{L^1\big(\bC^2,{\rm End}(\spaziorp_n)\big)^K}

\title[]{The Schwartz correspondence\\ for the complex motion group on $\bC^2$}

\author{Francesca Astengo}
\address{Dipartimento di Matematica, Universit\`a di Genova, Via Dodecaneso 35, 16146 Genova, Italy} 
\email{{\tt astengo@dima.unige.it}}

\author{Bianca Di Blasio}
\address{Dipartimento di Matematica e Applicazioni, Universit\`a di Milano Bicocca, Via Cozzi 53, 20125 Milano, Italy } 
\email{{\tt bianca.diblasio@unimib.it}}

\author{Fulvio Ricci}
\address{Scuola Normale Superiore, Piazza dei Cavalieri
7, 56126 Pisa, Italy } 
\email{{\tt fulvio.ricci@sns.it}}

\begin{document}

%\begin{flushright}{\tt{\jobname .tex}}\end{flushright}

\begin{abstract}
 If $(G,K)$ is a Gelfand pair, with $G$ a Lie group of polynomial growth and $K$ a compact subgroup of $G$, the Gelfand spectrum $\Sigma$ of the bi-$K$-invariant algebra $L^1(K\backslash G/K)$ admits natural embeddings into $\bR^n$ spaces as a closed subset.

For any such embedding, define $\cS(\Sigma)$ as the space of restrictions to $\Sigma$ of Schwartz functions on $\bR^n$. We call  {\it Schwartz correspondence} for  $(G,K)$ the property that the spherical transform is an isomorphism of $\cS(K\backslash G/K)$ onto $\cS(\Sigma)$.

In all the cases studied so far, Schwartz correspondence has been proved to hold true. These include all pairs with $G=K\ltimes H$ and $K$ abelian and a large number of pairs with $G=K\ltimes H$ and $H$ nilpotent. 

In this paper we study what is probably the simplest of the pairs with $G=K\ltimes H$, $K$ non-abelian and $H$ non-nilpotent, with $H=M_2(\bC)$, the complex motion group, and $K=U_2$ acting on it by inner automorphisms.

\end{abstract}

\maketitle

\baselineskip15pt

\section{Introduction}
\medskip

Let $(G,K)$ be a Gelfand pair, with $G$ a connected Lie group and $K$ a compact subgroup of it. 
By definition, this means that the convolution algebra $L^1(K\backslash G/K)$ of bi-$K$-invariant integrable functions on $G$ is commutative, or, equivalently, that the composition algebra $\bD(G/K)$ of $G$-invariant differential operators on $G/K$ is commutative.

 The Gelfand spectrum $\Sigma$ of $L^1(K\backslash G/K)$ is the space of bounded spherical functions on $G$ with the topology induced by the weak* topology on $L^\infty(K\backslash G/K)$. For each choice of a finite generating subset $\cD=\{D_1,\dots, D_\ell\}$ of $\bD(G/K)$, $\Sigma$ can be homeomorphically embedded into $\bC^\ell$, by assigning to each spherical function $\ph\in\Sigma$  the $\ell$-tuple $(\xi_1,\dots,\xi_\ell)$  if $D_j\ph=\xi_j\ph$ for $j=1,\dots,\ell$. The image $\Sigma_\cD$ of this embedding is closed \cite{FR}. 
\medskip

If $G$ has polynomial volume growth and the generators $D_j\in\cD$ are taken essentially self-adjoint, the eigenvalues are real, so that $\Sigma_\cD\subset\bR^\ell$. We refer to \cite{ADR3} for a presentation of Gelfand pairs of polynomial growth and the proofs of various preliminary results that will be needed in this paper.

We say that  {\it Schwartz correspondence} holds for a Gelfand pair $(G,K)$ of polynomial growth if the following property is satisfied:
\begin{enumerate}
\item[(S)] The spherical transform maps the bi-$K$-invariant Schwartz space $\cS(K\backslash G/K)$ isomorphically onto the space $\cS(\Sigma_\cD)$ of restrictions to $\Sigma_\cD$ of Schwartz functions on~$\bR^\ell$.
\end{enumerate}

This is an intrinsic property of the pair because it does not depend on the choice of the generating system $\cD$ \cite{ADR2, ADR3, FRY1}. It has been proved to be satisfied by all Gelfand pairs $(G,K)$ with polynomial growth on which it has been tested so far. These include compact pairs (i.e., with $G$ compact) \cite{ADR3}, various families of nilpotent pairs (i.e., with $G=K\ltimes N$ and $N$ nilpotent)\cite{ADR1, ADR2, FiR, FRY1, FRY2}, and pairs with $G=K\ltimes H$, where $K$ is abelian \cite{ADR3}.

In this paper we go one step ahead, considering one of the lowest dimensional examples of a pair $(K\ltimes H,K)$ of polynomial growth, with $K$ non-abelian and $H$ non-nilpotent.

\medskip

 One relevant feature of the pairs $(G,K)$ with $G$ a semidirect product, $G=K\ltimes H$, is that the bi-$K$-invariant algebra $L^1(K\backslash G/K)$ is isomorphic to the   algebra $L^1(H)^K$ of $K$-invariant integrable functions on $H$, and the operator algebra $\bD(G/K)$ is isomorphic to $\bD(H)^K$, the algebra of left- and $K$-invariant differential operators on $H$.
  It follows that saying that $(K\ltimes H,K)$ is a Gelfand pair is the same as saying that convolution on $H$ becomes commutative when restricted to $K$-invariant (integrable) functions. 
  
  The so-called {\it strong Gelfand pairs} fall into this picture. Given $G$ and $K\subset G$ compact, we denote by ${\rm Int}(K)$ the group of inner automorphisms of $G$ induced by  elements of $K$ and say that a function is $K$-central if it is ${\rm Int}(K)$-invariant. One says that $(G,K)$ is a strong Gelfand pair if the algebra $L^1(G)^{{\rm Int}(K)}$ of $K$-central integrable functions is commutative. From what has been said above, this is the same as saying that $$
 \big({\rm Int}(K)\ltimes G,{\rm Int}(K)\big)\cong \big(K\times G,{\rm diag}(K)\big)
 $$
 is a Gelfand pair.
 
 Obviously, a strong Gelfand pair is a Gelfand pair, so that we must distinguish between the two different notions of spherical functions, Gelfand spectrum and Schwartz correspondence. We refer to \cite{ADR3, RS} for the relative details.
 
 The list of strong Gelfand pairs of polynomial growth for which the Schwartz correspondence is not known to hold consists of the three families \cite{Y}
 $$
(SO_n\ltimes\bR^n, SO_n)\ ,\ n\ge 3\ ,\qquad (U_n\ltimes\bC^n,U_n)\ ,\ n\ge 2\ ,\qquad (U_n\ltimes H_n,U_n)\ ,\ n\ge 2\ ,
$$
where $H_n$ is the $(2n+1)$-dimensional Heisenberg group. We also adopt the notation $M_n(\bR)$ for $SO_n\ltimes\bR^n$, $M_n(\bC)$ for $U_n\ltimes\bC^n$, the real and  complex motion groups.
\medskip

The main result of this paper is that property (S) holds true for the strong pair $\big(M_2(\bC),U_2\big)$.
The algebra 
 $L^1\big(M_2(\bC)\big)^{{\rm Int}(U_2)}$ splits as the direct sum of the subalgebras
$L^1_\tau\big(M_2(\bC)\big)^{{\rm Int}(U_2)}$, $\tau\in\widehat U_2$, of $U_2$-central functions $f$ which are of $K$-type $\tau$ (see Section 2 for definitions). 

The following general principle allows to reduce verification of property (S) to a Schwartz extension property for single $K$-type components.

\begin{theorem}[{\cite[Prop. 5.2 and Thm. 7.1]{ADR3}}]\label{rendiconti}
Property (S) holds for a strong Gelfand pair $(G,K)$ of polynomial growth if and only if the following condition is satisfied:
\begin{enumerate}
\item[\rm(S')] given $f\in \cS(G)^{{\rm Int}(K)}$ and $N\in\bN$, for each $K$-type component $f_\tau$ of $f$, $\tau\in\widehat K$,   $\cG f_\tau$ admits a Schwartz extension $g_\tau^N$ such that $\|g_\tau^N\|_{(N)}$ is rapidly decaying in $\tau$.
\end{enumerate}
\end{theorem}

In our case the advantage in dealing with a single $K$-type $\tau$ is that the above extension problem can be equivalently reformulated as an extension problem for transforms of ${\rm End}(V_\tau)$-valued Schwartz functions on $\bC^2$, according to an isomorphism, quite standard in spherical analysis on semisimple groups \cite{ADR3, Camp, RS, War}, which intertwines convolution in the two settings. 

In this context standard Fourier analysis becomes available and the algebraic properties of the representation $\tau$ can be exploited. This is explained in Section 5. 

In the subsequent sections we prove 
\begin{enumerate}[(i)]
\item that Schwartz correspondence is satisfied by each $K$-type subalgebra (Sections 6 and subsections 7.1, 7.2);
\item that, denoting by $\cG_\tau$ the spherical transform restricted to functions of $K$-type $\tau$, $\cG_\tau\inv$ satisfies Schwartz norm inequalities that have moderate growth in $\tau$ (subsection 7.3);
\item that property (S') is satisfied (Section 8).
\end{enumerate}

Strictly speaking, this is more that what needs to be done. In fact, subsections 7.2 and 7.3 could be skipped by appealing to \cite[Prop. 5.1]{ADR3}, based upon \cite[Prop. 4.2.1]{Mar1}, stating that the implication $g\in\cS(\Sigma_\cD)\Rightarrow  \cG\inv g\in\cS(G)$ holds for general Gelfand pairs of polynomial growth. Application of this statement, however, would have still required some extra argument and, altogether,we find it preferable to  give a more self-contained proof.

In particular, we obtain that an ${\rm End}(V_\tau)$-valued $K$-equivariant Schwartz function~$F$ on $\bC^2$ can be expressed in a unique way as a sum
$$
F=\sum_{j=0}^n\mathbf D_n^jf_j\ ,
$$
where the $f_j$ are $K$-invariant scalar-valued Schwartz functions and $\mathbf D_n$ generates $\big(\bD(\bC^2)\otimes {\rm End}(V_\tau)\big)^K$ as a $\bD(\bC^2)^K$-module. 

This allows us to apply Schwartz correspondence for $\big(M_2(\bC),U_2\big)$ as an ordinary Gelfand pair, proved in \cite{ADR2}, by reducing the Schwartz extension problem to the single $f_j$.
However, the norm estimates that come along the way are not sufficient to guarantee that the extensions constructed in this way have the rapid decay in $\tau$ that is required in condition~(S').

To verify that this decay condition can be realized is the object of part (iii) of the proof. 
Given a $K$-central function $f$ on $M_2(\bC)$ with $K$-type components $f_\tau$, we then use a second procedure to extend the spherical transforms $\cG f_\tau$ of the $f_\tau$. It may look at this point that the previous construction was superfluous. The point is that the crucial step of this new procedure is the construction of an infinite jet at the origin which is compatible with the values  of $\cG f_\tau$. This amounts to proving, for each degree of homogeneity, that certain overdetermined linear systems are solvable, and this can be done by appealing to the fact that we already know that Schwartz extensions exist.

\bigskip

\section{Notation and definitions}
\medskip

In this section we collect the main notation and conventions that will be used in the paper, recalling some basic facts at the same time.

\subsection{Gelfand pairs and strong Gelfand pairs}\label{trasformate}\quad

Let $G$ be a  Lie group, $K$ a compact group and $\sigma$ an action of $K$ on $G$. We introduce the following notation.
\begin{itemize}
\item If  $X(G)$ is a space of scalar-valued functions on $G$, $X(G)^{\sigma(K)}$ denotes the subspace of $\sigma(K)$-invariant elements. We simply write $X(G)^K$ if there is no ambiguity on the action $\sigma$.
\item In particular, if $K\subseteq G$ and ${\rm Int}(K)$ denotes the group of conjugations by elements of~$K$, $X(G)^{{\rm Int}(K)}$ is the space of {\it $K$-central} functions on $G$. 
\item A function $f$ on $G$ is called {\it of $K$-type $\tau$}, with $\tau\in \widehat K$, if $f=f*_K(d_\tau\overline{\chi_\tau})$, where $d_\tau$ and $\chi_\tau$ are dimension and character of $\tau$. The symbol $X_\tau(G)^{{\rm Int}(K)}$ denotes the subspace of $K$-central functions of $K$-type $\tau$ in $X(G)$.
\item More generally, if  $V_\pi$ is a finite-dimensional representation space of $K$ and $X(G,V_\pi)$ is a space of $V_\pi$-valued functions on $G$, we denote by $X(G,V_\pi)^{\sigma\otimes\pi(K)}$ (or simply $X(G,V_\pi)^K$ when possible) the space of $K$-equivariant elements $F$ of $X(G,V_\pi)$, i.e., such that, for all $x\in G$ and $k\in K$, 
$F\big(\sigma(k)x\big)=\pi(k)F(x)$. 
\item By $\bD(G)$ we denote the algebra of left-invariant differential operators on $G$, and by $\bD(G)^{{\rm Int}(K)}\cong \fU(\fg)^{{\rm Ad}(K)}$ the subalgebra of those which are also invariant under ${\rm Int}(K)$. 
\item By $\bD(G/K)\cong  \fU(\fg)^{{\rm Ad}(K)}/\big( \fU(\fg)^{{\rm Ad}(K)}\cap\fU(\fg)\fk\big)$ we denote the algebra of $G$-invariant differential operators on $G/K$.
\end{itemize}

The following properties hold \cite{RS}:
\begin{enumerate}[(i)]
\item for each $\tau\in\widehat K$, $L^1_\tau(G)^{{\rm Int}(K)}$ is an algebra;
\item  $\sum_\tau L^1_\tau(G)^{{\rm Int}(K)}$ is dense in $L^1(G)^{{\rm Int}(K)}$  and given $f=\sum_\tau f_\tau$, $g=\sum_\tau g_\tau$ with $f_\tau,g_\tau\in L^1_\tau(G)^{{\rm Int}(K)}$, then
$$
f*g=\sum_\tau f_\tau*g_\tau\ ;
$$
in particular, $f_\tau*g_{\tau'}=0$ if $\tau\not\sim\tau'$;
\item with $\tau_0$ denoting the trivial representation of $K$, $L^1_{\tau_0}(G)^{{\rm Int}(K)}=L^1(K\backslash G/K)$.

\end{enumerate}

\begin{definition}
Let $G$ be a Lie group and $K$ a compact subgroup of $G$. 
\begin{enumerate}[\rm (a)]
\item $(G,K)$ is called a {\it Gelfand pair} if the algebra $L^1(K\backslash G/K)$ is commutative;
\item $(G,K)$ is called a {\it strong Gelfand pair} if the algebra $L^1(G)^{{\rm Int}(K)}$ is commutative;
\item for $\tau\in \widehat K$, $(G,K,\tau)$ is called {\it a commutative triple} if the algebra $L^1_\tau(G)^{{\rm Int}(K)}$ of $K$-central functions of $K$-type $\tau$ is commutative.
\end{enumerate}
\end{definition}

As a consequence of (i)-(iii), we have the following implications:

\be\label{implications}
\left.\begin{matrix}(G,K)\text{ strong Gelfand pair }\\ \big\Updownarrow\\(G,K,\tau)\text{ commutative triple for every }\tau\end{matrix}\right\} \Longrightarrow\left\{\begin{matrix}(G,K)\text{ Gelfand pair }\\ \big\Updownarrow\\(G,K,\tau_0)\text{ commutative triple }\end{matrix}\right.
\ee

Each of the three types of commutative structure listed in the previous definition has its own kind of {\it spherical functions}, 
defined as the normalized joint eigenfunctions of the appropriate differential operators and with the appropriate invariance properties. Precisely,
\begin{enumerate}[\rm (a')]
\item if $(G,K)$ is a Gelfand pair, the bi-$K$-invariant eigenfunctions of all operators in $\bD(G/K)$, taking value 1 at the identity element;
\item if $(G,K)$ is a strong Gelfand pair, the $K$-central eigenfunctions of all operators in $\bD(G)^{{\rm Int}(K)}$, taking value 1 at the identity element;
\item if $(G,K,\tau)$ is a commutative triple, the $K$-central eigenfunctions of $K$-type $\tau$ of all operators in $\bD(G)^{{\rm Int}(K)}$, taking value 1 at the identity element.
\end{enumerate}

The {\it bounded} spherical functions defined in (a'), resp. (b'), (c'), determine the multiplicative functionals on the corresponding $L^1$ algebra in (a), resp. (b), (c),  via the formula
$$
f\longmapsto \int_Gf(x)\ph(x\inv)\,dx\qquad \text{($\ph$  spherical).}
$$

In each case, the bounded spherical functions form the {\it Gelfand spectrum}  
of the corresponding $L^1$ algebra. Each Gelfand spectrum is endowed with the weak* topology induced from $L^\infty(G)$,  coinciding with the compact-open topology.

For given $f$, the map $\ph\longmapsto \int_Gf(x)\ph(x\inv)\,dx$, defined on the spectrum $\Sigma$, is the {\it spherical transform} of $f$ in the given structure.

If $\cD=\{D_1,\dots,D_\ell\}$ is a system of generators of the appropriate algebra of differential operators in (a')-(c'), we denote by $\Sigma_\cD\subset\bC^\ell$ the corresponding {\it embedded spectrum}, where each bounded spherical function is represented by the $\ell$-tuple of its eigenvalues w.r. to the elements of $\cD$. 

Then $\Sigma$ and $\Sigma_\cD$ are 
homeomorphic and $\Sigma_\cD$ is closed in $\bC^\ell$ \cite{FR}. If $G$ has polynomial growth and the $\ell$ generators are symmetric, then  $\Sigma_\cD\subset\bR^\ell$ \cite[Lemma~4.1]{ADR3}.
\medskip

 Assume that $(G,K)$ is a strong Gelfand pair, as we will do in the course of this paper. We denote by $\Sigma$
its spectrum and, for $\tau\in\widehat K$, we denote by $\Sigma^\tau$ the Gelfand spectrum of the commutative triple $(G,K,\tau)$. In particular, $\Sigma^{\tau_0}$ is the Gelfand spectrum of the underlying (non-strong) Gelfand pair $(G,K)$.

By \cite{RS,ADR3}, each $\ph\in\Sigma$ has a $K$-type, so that
$$
\Sigma=\bigcup_{\tau\in\widehat K}\Sigma^\tau\ ,
$$
where the union is disjoint and each term is open and closed.

By $\cG:L^1(G)^{{\rm Int}(K)}\longrightarrow C_0(\Sigma)$ we denote the spherical transform of the strong Gelfand pair~$(G,K)$. Then its restriction
$\cG_\tau: L^1_\tau(G)^{{\rm Int}(K)}\longrightarrow C_0(\Sigma^\tau)$
is the spherical transform of the commutative triple $(G,K,\tau)$, and $\cG_{\tau_0}$ the spherical transform of the (non-strong) Gelfand pair $(G,K)$.
\medskip

\subsection{Notation for $U_2$ and its irreducible representations} \label{s:rep}\quad

 First of all, we denote by $\tau_n$ the irreducible representation of $SU_2$ of dimension $n+1$ and by $\spaziorp_n$ the (abstract) representation space for $\tau_n$. We will often use the realization of $\spaziorp_n$ as the space $\cP^{(n,0)}(\bC^2)$ of holomorphic polynomials on $\bC^2$ that are homogeneous of degree $n$, with  
$$
\big[\tau_n(k)p\big](z)=p(k\inv z)\ ,\qquad  k\in SU_2\ .
$$

We then define, for $n\ge0$ and $m\in n+2\bZ$, the representation $\tau_{m,n}$ of $U_2$ on $\spaziorp_n$ such that, for $k=e^{i\theta}k'$ with $k'\in SU_2$,
$$
\tau_{m,n}(k)=e^{-im\theta}\tau_n(k')\ .
$$

For $\spaziorp_n=\cP^{(n,0)}(\bC^2)$, this takes the form
\[
%\label{taumn}
 \big[\tau_{m,n}(k)p\big](z)=e^{-im\theta}p({k'}\inv z)=(\det k)^{(n-m)/2}\,p(k^{-1}z)\ ,\qquad k\in U_2\ .
\]

In particular, 
\be\label{m=n}
\big[\tau_{n,n}(k)p\big](z)=p(k^{-1}z)\ ,\qquad  k\in U_2\ .
\ee

We set
$$
E=\big\{(m,n): n\ge0\ ,\ n-m\in2\bZ\big\}\ .
$$

 We also fix the basis of $\fs\fu_2$
$$
\baseuno =\begin{bmatrix}i&0\\0&-i\end{bmatrix}
\ ,\qquad\basedue=\begin{bmatrix}0&1\\-1&0\end{bmatrix}
\ ,\qquad \basetre=\begin{bmatrix}0&i\\i&0\end{bmatrix} 
%&\basequattro=\begin{bmatrix}i&0\\0&i\end{bmatrix}
\ ,
$$
and set $X_4=iI$ to complete a basis of $\fu_2$.

Choosing $\ft=\bR \baseuno$ as a maximal toral subalgebra of $\fs\fu_2$, the elements $\basedue\pm i\basetre$ of $\fs\fu_2^\bC$ are root vectors, relative to the roots $\mp2i$ respectively.

Moreover, each monomial $z_1^jz_2^{n-j}$, $j=0,\dots,n$ is a weight vector for the representation $\tau_n$, with weight $(n-2j)i$.
With respect to the $K$-invariant  Fischer inner product
\[
%\label{FischeronVn}
\langle p,q\rangle_n=\frac{1}{n!}\, p(\de_z)q^*
\qquad\text{where}\quad q^*(z)=\overline{q(\bar z)}\ ,
\]
the normalized monomials
\be\label{basisofVn}
e^j_n(z)=\binom{n}{j}^{1/2}\,z_1^jz_2^{n-j} \qquad j=0,\ldots , n
\ee
form an orthonormal basis of $\spaziorp_n$.

\bigskip

\section{The complex motion group $M_2(\bC)$ and the strong Gelfand pair $\big(M_2(\bC),U_2\big)$}

The complex motion group $G=M_2(\bC)$ is the semidirect product $U_2\ltimes\bC^2$, where the action of $U_2$  on $\bC^2$ is the natural one. 

In the sequel the symbols $G$ and $K$  will more often be used to denote $M_2(\bC)$ and $U_2$ respectively. Only in a few occasions they will stand for elements of a general pair $(G,K)$. The context will remove any possible ambiguity.

\begin{proposition}
The pair $(G,K)=\big(M_2(\bC),U_2\big)$ is a strong Gelfand pair.
\end{proposition}

\begin{proof} It follows from \cite[Thm.~10.1, see also Cor.~10.4]{RS} that $\big(M_2(\bC),U_2,\tau_{m,n}\big)$ is a commutative triple for every $m,n$. By \eqref{implications}, this gives the conclusion.
\end{proof}

We write elements of $G$ as pairs $(k,z)\in U_2\times \bC^2$ with product
$$
(k,z)(k',z')=(kk',z+kz')\ .
$$

The adjoint action of $U_2$  on the Lie algebra $\fg=\fm_2(\bC)\cong\fu_2\times\bC^2$ is
$$
{\rm Ad}(k)(U,z)=(kUk\inv,kz)\ ,
$$
which splits $\fg$ as $\fs\fu_2\times\bR(iI)\times\bC^2$. We decompose $U\in\fu_2$ as $X+itI$ with $X\in\fs\fu_2$ and $it=\tr U$ with $t\in\bR$.

The algebra of ${\rm Ad}(U_2)$-invariant polynomials on $\fg$ is freely generated by the four polynomials
$$
 p_1=|z|^2\ ,\qquad p_2=z^*Xz\ ,\qquad p_3=\det X=|X|^2\ ,\qquad p_4=t\ ,
$$
where $z$ is represented by the column vector $\begin{bmatrix}z_1\\z_2\end{bmatrix}$, cf. \cite[Thm. 7.5]{FRY1}. Writing 
$$
X=\begin{bmatrix}ix_1&x_2+ix_3\\-x_2+ix_3&-ix_1\end{bmatrix}=x_1X_1+x_2X_2+x_3X_3\ ,
$$
the second polynomial becomes
$$
p_2=ix_1\big(|z_1|^2-|z_2|^2\big)-x_2\big (z_1\bar z_2-\bar z_1z_2\big)+ix_3\big(\bar z_1z_2+z_1\bar z_2\big)\ .
$$

In order to  obtain generators $D_j$ of $\bD(G)^{{\rm Int}(K)}$, we apply the symmetrization $\la'$ in \cite[formula (2.4)]{RS} to the $p_j$'s. Allowing multiplication by scalar coefficients in order to obtain symmetric operators, positive when they have a sign, we set
$$
D_1=\Delta_z\ ,\qquad D_3=\Cas\ ,\qquad D_4=iX_4\ ,
$$
where $\displaystyle \Delta_z=-4\,\sum_{j=1,2}\big(\de_{z_j}\de_{\bar {z_j}} +\de_{\bar {z_j}}\de_{{z_j}}\big)$ and $\Cas=-X_1^2-X_2^2-X_3^3$ is the Casimir operator on~$SU_2$.

To obtain $D_2$, observe that, if $p(x,z)=\sum_jq_j(x)r_j(\RE z,\IM z)$, the symmetrization $\la'$ on $G$ is symmetrization on $U_2$ followed by symmetrization on $\bC^2$ on each summand, i.e.,
\be\label{lambda'}
\la'(p)f(k,z)=\sum_jr_j(\de_{\RE w},\de_{\IM w})_{|_{w=0}}q_j(\de_x)_{|_{x=0}} f\big((k,z)(e,w)(\exp_Kx,0)\big)\ .
\ee

Therefore
$$
 D_2=i\big(\Delta_{z_2}-\Delta_{z_1}\big)X_1-4\big(\de_{\bar z_1}\de_{z_2}-\de_{z_1}\de_{\bar z_2}\big)X_2+4i\big(\de_{z_1}\de_{\bar z_2}+\de_{\bar z_1}\de_{z_2}\big)X_3\ .
$$

\begin{remark}\label{SU2}\label{SU2}
It is worth noticing that $\big(SM_2(\bC),SU_2\big)$, where $SM_2(\bC)=SU_2\ltimes\bC^2$, is a Gelfand pair but not a strong one. To see this, one can use the representation theoretic argument in \cite[Cor. 10.4]{RS} or, alternatively, observe that the polynomial
$$
q(X,z)={}^t\!zJX z\ ,\qquad \text{where }J=\begin{bmatrix}0&1\\-1&0\end{bmatrix}
$$
on $\fs\fu_2\times\bC^2$ is $SU_2$-invariant, but its symmetrization 
$$
\la'(q)=-2i\de_{z_1}\de_{z_2}X_1-(\de_{z_1}^2+\de_{z_2}^2)X_2+i(\de_{z_1}^2-\de_{z_2}^2)X_3
$$
 does not commute with $D_2$.
\end{remark}

\bigskip

\section{ $K$-type subalgebras and ${\rm End}(\spaziorp_n)$-valued spherical analysis}\label{EndVn}

It is a general fact that, given a Lie group $G$ together with a compact subgroup $K$ and a representation $\tau\in\widehat K$,  there is a one-to-one correspondence between $K$-central scalar-valued functions $f$ on $G$ of $K$-type $\tau$ and bi-$\tau$-equivariant integrable functions $F$ from $G$ to ${\rm End}(\spaziorp_\tau)$, i.e., verifying the identity
\[
%\label{equivariance}
F(k_1xk_2)=\tau(k_2\inv)F(x)\tau(k_1\inv)\ ,\qquad\forall\,k_1,k_2\in K\ .
\]

This correspondence is given by
\be\label{f<->Fgeneral}
f\longmapsto F(x)=\int_Kf(xk)\tau(k)\,dk\ ,\qquad F\longmapsto f(x)=d_\tau\tr F(x)\ ,
\ee
 preserves integrability and respects convolution, once this is defined on ${\rm End}(\spaziorp_\tau)$-valued functions as
$$
F_1*F_2(x)=\int_G F_2(y\inv x)\,F_1(y)\,dy\ .
$$

Basic notions and results can be found in \cite[vol. II, Ch. 6]{War} and in \cite{Camp, RS}. 

\medskip

In our case, since $G=U_2\ltimes\bC^2$, we can reduce matters to ${\rm End}(\spaziorp_n)$-valued functions on $\bC^2$.
In fact, the restriction $F_\flat$ of a bi-$\tau_{m,n}$-equivariant $F$ on $G$ to $\{e\}\times\bC^2$ completely determines $F$ and satisfies the identity
\be\label{flat}
F_\flat(kz)=F\big((k,0)(e,z)(k\inv,0)\big)=\tau_{m,n}(k)F_\flat(z)\tau_{m,n}(k\inv)\ ,\qquad\forall\,k\in U_2\ .
\ee

Conversely, given $F$ on $\bC^2$ such that
\be\label{tau-equivariance1}
F(kz)=\tau_{m,n}(k)F(z)\tau_{m,n}(k\inv)\ ,\qquad\forall\,k\in U_2\ ,
\ee
 the function
\be\label{sharp}
F^\sharp(k,z)=\tau_{m,n}(k\inv)F(z)
\ee
is bi-$\tau_{m,n}$-equivariant on $G$.

Moreover, the two maps $F\longmapsto F_\flat$ and $F\longmapsto F^\sharp$ are mutually inverse and respect convolution.

\begin{remark}\label{tildetaun}
It must be observed that the equivariance condition \eqref{tau-equivariance1}
does not depend on~$m$. 
This reflects the fact that, in  the tensor products 
$$
\tau'_{m,n}\otimes\tau_{m,n}=\tau_{-m,n}\otimes\tau_{m,n}\ ,
$$
the center of $U_2$ acts trivially, so that they all define the  same representation of $U_2$, denoted by $\tilde\tau_n$, on ${\rm End}(\spaziorp_n)\cong V'_n\otimes \spaziorp_n$.

It also reflects the fact that, for fixed $n$,
the $L^1_{\tau_{m,n}}(G)^{{\rm Int}(K)}$ algebras are all isomorphic, since 
\be\label{m-iso}
f\in L^1_{\tau_{m,n}}(G)^{{\rm Int}(K)}\ \Longleftrightarrow\ (\det k)^{(m-m')/2}f\in L^1_{\tau_{m',n}}(G)^{{\rm Int}(K)}\ ,
\ee
and multiplication by $(\det k)^{(m-m')/2}$ respects convolution. 

It is important to keep track, however, of the fact that $\tilde\tau_n$ is a representation of $U_2$ and not just of $SU_2$: e.g., if the action is restricted to $SU_2$, the equivariance condition \eqref{tau-equivariance1} does not include the identity $F(e^{i\theta}z)=F(z)$ for $n>0$ and commutativity is lost, n accordance with Remark \ref{SU2}.
\end{remark}

We say that an ${\rm End}(\spaziorp_n)$-valued function on $\bC^2$ is $K$-equivariant
if
\be\label{tau-equivariance}
F(kz)=\tilde\tau_n(k)F(z)=\tau_{m,n}(k)F(z)\tau_{m,n}(k\inv)\ ,\qquad\forall\,k\in K\ ,
\ee
and denote by $\Lmatr$, and similarly for other function spaces, the subspace of {$K$-equivariant} functions.

Combining \eqref{f<->Fgeneral}-\eqref{sharp} together, we have the following.

\begin{lemma}\label{f<->F}
The two maps
\begin{align*}
%\label{Amn}
\begin{array}{rcrcrcl}A_{m,n}&:&f(k,z)&\longmapsto &F(z)&=& \tau_{m,n}\big(f(\cdot, z)\big) \ = \ \int_Kf(k,z)\tau_{m,n}(k)\,dk\\ 
A_{m,n}\inv&:&F(z)&\longmapsto &f(k,z)&=&(n+1)\tr\big(\tau_{m,n}(k\inv)F(z)\big)
\end{array}
\end{align*}
establish a one-to-one correspondence between $K$-central functions $f$ on $G$ of $K$-type $\tau_{m,n}$ and $K$-equivariant ${\rm End}(\spaziorp_n)$-valued functions on $\bC^2$.

In particular, $A_{m,n}$ is an isomorphism of algebras from $L^1_{\tau_{m,n}}(G)^{{\rm Int}(K)}$ onto $L^1\big(\bC^2,{\rm End}(\spaziorp_n)\big)^K$ and $\sqrt{n+1}A_{m,n}$ is unitary from $L^2_{\tau_{m,n}}(G)^{{\rm Int}(K)}$ onto $L^2\big(\bC^2,{\rm End}(\spaziorp_n)\big)^K$, where, for $F\in L^2\big(\bC^2,{\rm End}(\spaziorp_n)\big)$,
\be\label{norma2}
\|F\|_2^2=\int_{\bC^2}\big\|F(z)\big\|^2_{HS}\,dz\ .
\ee
\end{lemma}

\begin{remark} \label{rem:fAmnf}
The relation between $f$ and $A_{m,n}f$ can be made more explicit observing the following facts.

Let $G=K\ltimes H$ with $K$ compact. Then a function $f$ on $G$ is of $K$-type $\tau\in\widehat K$ if and only if it can be written as
$$
f(k,h)=d_\tau\sum_{i,j}f_{ij}(h)\overline{\tau_{ij}(k)} ,
$$
where, relative to some orthonormal basis of $\spaziorp_\tau$, $\tau(k)=\big(\tau_{ij}(k)\big)_{i,j}$  and $F(h)=\big(f_{ij}(h)\big)_{i,j}$.
The right-hand side equals $d_\tau\tr\big(\tau(k\inv)F(h)\big)$ and $K$-centrality of $f$ is equivalent to $K$-equivariance of $F$.
\end{remark}

\medskip

We denote by $\big(\bD(\bC^2)\otimes {\rm End}(\spaziorp_n)\big)^K$
the algebra of ``${\rm End}(\spaziorp_n)$-valued'' differential operators on $\bC^2$ which commute with translations and with the action of $\grpU$ on smooth ${\rm End}(\spaziorp_n)$-valued functions $F$ given by
$$
k:F\longmapsto F^k(h)=\tau_{m,n}(k)F(k\inv\cdot h)\tau_{m,n}(k)\inv\ .
$$

We recall the linear symmetrization $\la':\cP(\fu_2\times\bC^2)\longrightarrow \bD(G)$ defined in \eqref{lambda'} for polynomials in separate variables. The following statement, proved in \cite[Cor. 2.3 and Prop.~2.4]{RS}, gives the conjugation formula for $\la'(p)$ under $A_{m,n}$.

\begin{lemma}\label{symmetrization}
Let $p(x,z)=\sum_jq_j(x)r_j(z)$ be a polynomial on $\fu_2\times\bC^2$. Defining  $\check q_j(x)=q_j(-x)$, we have the identity
$$
A_{m,n}\la'(p)A_{m,n}^{-1}=\sum_j r_j(\de)\otimes d\tau_{m,n}\big(\la_K(\check q_j)\big)\ .
$$

Conjugation by $A_{m,n}$ is a homomorphism of $\bD(G)^{{\rm Int}(K)}$   onto $\big(\bD(\bC^2)\otimes {\rm End}(\spaziorp_n)\big)^{K}$ and its kernel consists of the operators which vanish on functions of $K$-type $\tau_{m,n}$.
\end{lemma}

In particular, this lemma
establishes the correspondence 
\bea\label{AmnD}
 D_1& \longleftrightarrow  \Delta_z\otimes I\\
 D_2& \longleftrightarrow i\big(\Delta_{z_2}-\Delta_{z_1}\big)\otimes d\tau_n(X_1)-4\big(\de_{\bar z_1}\de_{z_2}-\de_{z_1}\de_{\bar z_2}\big)\otimes d\tau_n(X_2)\\
  &\qquad\qquad\qquad\qquad\qquad\qquad\qquad +4i\big(\de_{z_1}\de_{\bar z_2}+\de_{\bar z_1}\de_{z_2}\big)\otimes d\tau_n(X_3)\\
 D_3& \longleftrightarrow d\tau_{m,n}(\Cas)=1\otimes(n^2+2n)I \\
  D_4& \longleftrightarrow d\tau_{m,n}(iX_4)=1\otimes mI\ .
\eea

We can then conclude that $\big(\bD(\bC^2)\otimes {\rm End}(\spaziorp_n)\big)^{K}$ is generated by the operators~$A_{m,n}D_1A_{m,n}^{-1}$ and $A_{m,n}D_2A_{m,n}^{-1}$, which we simply denote by $\Delta_z$ and $\DD_n$ respectively (since $\baseuno,\basedue,\basetre\in\fs\fu_2$, the operator $A_{m,n}D_2A_{m,n}^{-1}$ does not depend on $m$). We remark that they are not free generators.

\begin{lemma} \label{l:baseoperatdiff}
The operators $\Delta_z^j\DD_n^k$ with $j\in\bN$ and $0\le k\le n$ form a basis of $\big(\bD(\bC^2)\otimes {\rm End}(\spaziorp_n)\big)^{K}$.
\end{lemma}

\begin{proof}
Taking Fourier transform in $z$, the symbol $\widehat\DD_n$ of $\DD_n$ can be expressed as an $(n+1)\times(n+1)$ matrix with polynomial entries in the dual variable $\zeta$. The coefficients $q_{n,k}$ of the characteristic equation 
\be\label{characteristic}
\det\big(\la I-\widehat \DD_n(\zeta)\big)=\la^{n+1}+\sum_{k=0}^n \la^kq_{n,k}(\zeta)=0
\ee 
are $K$-invariant polynomials, hence polynomials $p_{n,k}$ in $|\zeta|^2$.
Applying the Cayley-Hamilton theorem and undoing Fourier transform,
$$
\DD_n^{n+1}=\sum_{k=0}^n \DD_n^kp_{n,k}(\Delta_z)\ .
$$

It remains to prove that $\widehat \DD_n$ 
 does not solve any equation of smaller degree in $\la$ than \eqref{characteristic}. This follows from the fact that $\widehat\DD_n(0,1)=i d\tau_n(X_1)$ has $n+1$ distinct eigenvalues.
\end{proof}

Following \cite{RS, ADR3}, the characters of $L^1_{\tau_{m,n}}(G)^{{\rm Int}(K)}$ are given by integration against the bounded spherical functions of $K$-type $\tau_{m,n}$.

Applying Lemma \ref{f<->F}, one can see that the characters of $L^1\big(\bC^2,{\rm End}(\spaziorp_n)\big)^K$ have the form
\be\label{matrix-sph}
F\longrightarrow (n+1)\int_{\bC^2}\tr\big(F(z)A_{m,n}\ph(-z)\big)\,dz\ ,
\ee
where $\ph$ is a spherical function of the strong Gelfand pair $(G,K)$ of $K$-type $\tau_{m,n}$.

\begin{definition}\label{defsferichematr}
We call ${\rm End}(\spaziorp_n)${\it -valued spherical functions}\footnote{More frequently, cf. \cite[vol. II, ch. 6]{War}, these are called $\tau_{m,n}$-spherical functions and the scalar-valued ones {\it spherical trace} functions.
} the functions
$$
\Phi(z)=(n+1)^2(A_{m,n}\ph)(z)\ ,
$$
where $\ph$ is a spherical function of the strong Gelfand pair $(G,K)$ of $K$-type $\tau_{m,n}$.
\end{definition}

The factor $(n+1)^2$ is due to the normalization $\Phi(0)=I$.

For what has been said above, the definition does not depend on $m$. Moreover,
\begin{itemize}
\item the ${\rm End}(\spaziorp_n)$-valued spherical functions are characterized by the property of being the joint eigenfunctions~$\Phi$ of $\Delta_z$ and $\DD_n$ which are $K$-equivariant, with $\Phi(0)=I$;
\item the bounded ${\rm End}(\spaziorp_n)$-valued spherical functions are the $K$-equivariant functions $\Phi$ which define nontrivial multiplicative functionals on $L^1\big(\bC^2,{\rm End}(\spaziorp_n)\big)^K$ through the formula
\be\label{Gn}
F\longmapsto \frac1{n+1}\int_{\bC^2}\tr\big(F(z)\Phi(-z)\big)\,dz\ .
\ee
\end{itemize}

\begin{remark}\label{type0}
For $n=0$, and in particular for the trivial representation $\tau_{0,0}$, Lemma \ref{f<->F} establishes the trivial fact, valid for every semidirect product $G=K\ltimes H$, that $K$-invariant (scalar) functions on $H$ coincide with  restrictions to $\{e_K\}\times H$ of bi-$K$-invariant functions on $G$ and that, via this identification, $L^1(K\backslash G/K)$ is isomorphic to $L^1(H)^K$. Similarly \eqref{Gn} is coherent with the fact that $A_{m,0}$ identifies the spherical functions of the (non-strong) Gelfand pair $(G,K)$ with the ${\rm End}(V_0)$-valued spherical function on $\bC^2$.
\end{remark}

\bigskip

\section{Embedded spectra and Schwartz correspondence}

Given a bounded spherical function $\ph$ of the strong Gelfand pair $(G,K)=\big(M_2(\bC),U_2\big)$, we denote by $\xi(\ph)$ the quadruple $(\xi_1,\xi_2,\xi_3,\xi_4)$ of eigenvalues of $\ph$ with respect to $D_1,D_2,D_3,D_4$ respectively.

With $\cD=\{D_1,D_2,D_3,D_4\}$, we denote by $\Sigma_\cD$ the ``embedded'' spectrum of the strong Gelfand pair,
$$
\Sigma_\cD=\big\{\xi(\ph):\ph\in\Sigma\big\}\ .
$$

Conversely, given $\xi\in\Sigma_\cD$, we denote by $\ph_\xi$ the spherical function such that $\xi(\ph_\xi)=\xi$.

It follows from \cite[Lemma 4.1]{ADR3} that $\Sigma_\cD\subset\bR^4$ and is closed. By \eqref{AmnD}, if $\ph$ is of $K$-type $\tau_{m,n}$ then
$$
\xi(\ph)=(\xi_1,\xi_2,n^2+2n,m)\ .
$$

In accordance with \cite[Sect. 6]{ADR3}, we set
$$
\Sigma_\cD^n=\big\{(\xi_1,\xi_2):(\xi_1,\xi_2,n^2+2n,m)\in\Sigma_\cD\big\}\ ,
$$
which is independent of $m$ by \eqref{m-iso}. Then, recalling that $E=\big\{(m,n):n\ge0\,,\, n-m\in2\bZ\big\}$,
\be\label{dec-sigma}
\Sigma_\cD=\bigcup_{(m,n)\in E}\Sigma_\cD^n\times\big\{(n^2+2n,m)\big\}\ .
\ee

Therefore, coherently with the notation in subsection~\ref{trasformate}, we shall consider    the spherical transform $\cGmn f$, of $f\in L^1_{\tau_{m,n}}(G)^{{\rm Int}(K)}$,  as a function defined   on   $\Sigma_\cD^n\subset\bR^2$.

In order to prove condition (S') in Theorem \ref{rendiconti}, in the next sections we prove Schwartz correspondence for each $K$-type spherical transform $\cG_{\tau_{m,n}}$. 

In doing so, it is convenient to adopt the ${\rm End}(\spaziorp_n)$-valued model, replacing $\cS_{\tau_{m,n}}(G)^{{\rm Int}(K)}$ with $\cS\big(\bC^2,{\rm End}(\spaziorp_n)\big)^K$. 
This allows us to take advantage of the algebraic structure of ${\rm End}(\spaziorp_n)$ and to completely disregard the parameter $m$.

By Lemma~\ref{symmetrization} and \cite[Cor. 2.3 and following remarks]{RS}
the first two components
$\xi_1,\xi_2$ of $\xi(\ph)$, with $\ph$ spherical of $K$-type $\tau_{m,n}$,  are the eigenvalues of $\Phi=(n+1)^2A_{m,n}\ph$ 
under the action of $\Delta_z$ and $\DD_n$ respectively.

For better clarity, we will use the slightly different notation $\cG_n$ for the spherical transform of $L^1\big(\bC^2,{\rm End}(\spaziorp_n)\big)^K$. In the following sections, ${\rm End}(V_n)$-valued spherical functions $\Phi$ will be labeled with parameters identifying the pair of its eigenvalues $(\xi_1,\xi_2)\in\Sigma^n_\cD$.

\bigskip

\section{${\rm End}(\spaziorp_n)$-valued functions}

In this section we start studying ${\rm End}(\spaziorp_n)$-valued $K$-equivariant functions.
From Lemma~\ref{l:baseoperatdiff}, applying the Fourier transform in the $\bC^2$ variable, 
one can deduce that a basis for $K$-equivariant polynomials is given by 
\[
\{z\mapsto |z|^{2j} \,\widehat\DD_n^k(z)\, :\, j\in \bN,\quad k=0,\ldots, n\}\ .
\]
Our first task will be to replace this basis with one compatible with the decomposition 
of ${\rm End}(\spaziorp_n)$ into its $\tilde\tau_n$-invariant irreducible subspaces.

\subsection{Decomposition of ${\rm End}(\spaziorp_n)$}
\quad

Restricting to $SU_2$ the representation $\tilde\tau_n$ defined in Remark \ref{tildetaun}, as we can do by triviality on the center, we have
$$
\tilde\tau_n=\tau'_{n}\otimes\tau_{n}\,\sim\,\tau_{2n}\oplus \tau_{2n-2}\oplus\cdots\oplus\tau_2\oplus\tau_0\ ,
$$
and, correspondingly, 
\be\label{Pnn}
{\rm End}(\spaziorp_n) 
=\spaziW^n_n\oplus \spaziW^{n-1}_{n}\oplus\cdots\oplus \spaziW^0_n\ ,
\ee
where $\mathrm{dim}\,\spaziW^\ell_n=2\ell+1$.

\medskip
Since 
$$
B_n^1=id\tau_n(\baseuno )={\rm diag}(-n,\dots,-n+2\ell,\dots,n),
$$
has distinct eigenvalues, any diagonal matrix $B={\rm diag}(b_0,b_1,\dots,b_n)$, can be written as  a polynomial in $B_n^1$. Indeed, there is a unique polynomial $p$ of degree at most $n$ such that $p(-n+2\ell)=b_\ell$, $\ell=0,\ldots, n$. Then $B=p\big(B_n^1\big)$.

\begin{lemma}\label{diagonalWnj}
For every $\ell=0,\ldots ,n$, let $\diagonali^\ell_n$ be the subspace of diagonal matrices in $\spaziW^\ell_n$.
\begin{enumerate}[\rm(i)]
\item 
The subspace $\diagonali^\ell_n$ is one-dimensional. 
\item The subspace $\diagonali^1_n$ consists of the scalar multiples of $B_n^1$. 
\item For general~$\ell$, there is a monic polynomial $q_n^\ell$ of degree $\ell$ such that $\diagonali^\ell_n$ consists of the scalar multiples of the matrix
\begin{equation}\label{e:matriciBj}
B_n^\ell=q_n^\ell(B_n^1).
\end{equation}
\item For a polynomial $p$ of degree at most $n$, $p(B_n^1)\in \sum_{\ell\le j}\spaziW_n^\ell$ if and only if $\deg(p)\le j$.
\end{enumerate}
\end{lemma}

\begin{proof}
Since $d\tau_n(\baseuno )$ has distinct eigenvalues, diagonal matrices are those that commute with $d\tau_n(\baseuno )$.  
Since  $\tilde\tau_{n}$ restricted to $\spaziW^\ell_n$ is equivalent to $\tau_{2\ell}$, it contains the null weight with multiplicity one, and this proves (i).
 
Consider now the $\tilde\tau_n$-invariant subspace $\spaziW$ generated by $d\tau_n(\baseuno )$. Then
$$
\spaziW=\spaz_\bC\big\{\tau_n(k)d\tau_n(\baseuno )\tau_n(k)\inv\, :\,k\in \grpSU\big\}=d\tau_n(\fs\fu_2^\bC) \ ,
$$
which is a 3-dimensional invariant subspace. So it must coincide with the component $\spaziW_n^1$ in~\eqref{Pnn}. This proves (ii). 

The statements  (iii) and (iv) are trivial for $\ell=0$ and have been proved above for $\ell=1$. 
For  $\ell\geq 2$, item (iv) follows by induction from the fact that $\tilde\tau_n|_{\spaziW_n^{\ell-1}}\sim \tau_{2\ell-2}$ and the decomposition
\[
\tau_{2\ell-2}\otimes \tau_2\sim \tau_{2\ell}\oplus \tau_{2\ell-2} \oplus \tau_{2\ell-4}
\]
in irreducible summands with multiplicity one for  $\tau_{2\ell}$.
Finally, if $p(t)=t\, q_n^{\ell-1}(t)$, then $p$ is monic, has degree $\ell$ and $p(B_n^1)$ is a diagonal matrix in 
$\spaziW_n^\ell\oplus\cdots\oplus \spaziW_n^0$ with a nontrivial component in $\spaziW_n^\ell$, that we call $q_n^\ell(B_n^1)$. Hence $(p-q_n^\ell)(B_n^1)$ is in $\sum_{j\in L} \spaziW_n^j$, which, by the inductive hypothesis, is a polynomial in $B_n^1$ of degree at most $\ell-1$. Then $q_n^\ell$ has degree $\ell$ and its leading term is the same as $p$, so it is monic. 
\end{proof}

\subsection{Equivariant polynomials}
\quad

Suppose now $P:\bC^2\to {\rm End}(\spaziorp_n)$ is a $\grpU$-equivariant polynomial. Then the homogeneous component of $P$ of bi-degree $(d_1,d_2)$
 is also equivariant, and trivial if $d_1\ne d_2$.

Assume therefore that $P\in\cP^{d,d}$, i.e., homogeneous of bi-degree $(d,d)$. By homogeneity, $P$ is uniquely determined by its restriction to the unit sphere and, by equivariance,
\bea\label{P(base)}
P(z)&= \tau_n(k_{z})P(\base)\tau_n(k_{z}\inv) \qquad \forall z\in \bC^2\ ,\ |z|=1\ ,\\
P(\base)&=P(e^{i\theta} \base) \qquad \forall \theta\in \bR\ ,\\
P(\base)&=\tau_n(\exp(t \baseuno )) \,P (\base)\, \tau_n(\exp(t \baseuno ))\inv\qquad \forall t\in \bR\ 
\eea
where $\base$ is the base point $(0,1)$ and $k_z\in \grpSU$ is such that $k_z\base=z$.
The diffeomorphism 
\begin{equation}\label{e:id-sfera-su2}
 z=(z_1,z_2)\longmapsto k_z=
 \begin{bmatrix}
\bar z_2&z_1\\-\bar z_1&z_2
\end{bmatrix}
\end{equation}
identifies  the  unit sphere  $S^3$  with the group $\grpSU$.

We conclude that
the matrix $B=P(\base)$ is diagonal 
and, setting $z'=z/|z|$ for $z\ne(0,0)$, 
\be\label{P->B}
P(z)=|z|^{2d}\tau_n(k_{z'})B\tau_n(k_{z'})^*\ .
\ee

Conversely, given $B={\rm diag}(b_0,b_1,\dots,b_n)$, we will determine for what values of $d$ formula~\eqref{P->B} defines a polynomial. The answer to this question goes together with the issue of describing the equivariant polynomials taking values in a given $\spaziW_n^\ell$. In this respect the following remarks are quite obvious, after Lemma \ref{diagonalWnj}, for an ${\rm End}(\spaziorp_n)$-valued equivariant polynomial $P$:
\begin{itemize}
\item $P$ can be uniquely decomposed as the sum of $\spaziW_n^\ell$-valued ones;
\item $P$ takes values in $\spaziW_n^\ell$ if and only if $P(\base)\in \spaziW_n^\ell$.
\end{itemize}

 For $B={\rm diag}(b_0,b_1,\dots,b_n)$ we denote by $d(B)\le n$ the degree  of the polynomial $p$ such that $B=p(B_n^1)$. 
 From Lemma~\ref{diagonalWnj}, it follows that $d(B)\le j$ if and only if $B\in \sum_{\ell\le j}\spaziW_n^\ell$.
 
 For $B$ as above,
 and $d$ in $\bN$, we set
 $$
Q_{B}^d(z)= |z|^{2d}\tau_n(k_{z'})B\tau_n(k_{z'})^*\ ,\qquad z=|z|z'\ .
$$

\begin{lemma}\label{l:estensione-matricediag}
\quad
\begin{enumerate}[\rm(i)]
\item For a diagonal $B$, the function $Q^d_{B}$
can be continued to an ${\rm End}(\spaziorp_n)$-valued equivariant polynomial if and only if $d\ge d(B)$. In this case, the polynomial $Q^d_{B}$ is homogeneous of bi-degree $(d,d)$.
\item Every $\spaziW_n^\ell$-valued equivariant polynomial has the form $ p\big(|z|^2\big)Q^\ell_{B_n^\ell} (z)$, where $p$ is a scalar-valued polynomial in one variable.
\end{enumerate}
\end{lemma}

\begin{proof}
(i) Since any diagonal matrix $B$ is a linear combination of the matrices $B_n^\ell$, it is enough to treat the case where $B=B_n^\ell=q_n^\ell(B^1_n)$.

Assume that, for a given $d$, $Q^d_{B_n^\ell}$ extends to a polynomial.
Recalling that $\spaziW_n^\ell\sim \spaziorp_{2\ell}$ and denoting by $\mathcal{I}_n^\ell:\spaziW_n^\ell\to \spaziorp_{2\ell}$ a unitary intertwining operator,
let $\tilde Q_d$ be the $\spaziorp_{2\ell}$-valued polynomial defined by the rule 
\[
\tilde Q_d(z)=\mathcal{I}_n^\ell(Q^d_{B_n^\ell}(z))\qquad \forall z\in \bC^2\ .
\]

Then $\tilde Q_d$ is $\tau_{2\ell}$-equivariant and it suffices to prove the necessity of the condition 
$d\ge \ell$ for $\tilde Q_d$. By \eqref{P(base)},  third line, 
\[
\tilde Q_d(\base)=\tau_{2\ell}(\exp{t\baseuno })\tilde Q_d(\base)\ ,
\]
so that $\tilde Q_d(\base)$ is a 0-weight vector for $\tau_{2\ell}$. In the polynomial model of Section~\ref{s:rep}, $\tilde Q_d(\base)\in\spaziorp_{2\ell}$ has then the form
\[
\big[\tilde Q_d(\base)\big](w)=c\, w_1^\ell\, w_2^\ell\ 
\]
for some constant $c$. 
It follows that, if $|z'|=1$,
\[
\big[\tilde Q_d(z')\big](w)= \big[\tau_{2\ell}(k_{z'})\tilde Q_d(\base)\big](w)=c\, (z'_2w_1-z'_1w_2)^\ell(\bar z'_1w_1+\bar z'_2w_2)^\ell\ ,
\]
and the homogeneous extension of bi-degree $(d,d)$ with $z$ in $\bC^2$ is
$$
\big[\tilde Q_d(z)\big](w)= c|z|^{2(d-\ell)}(z_2w_1-z_1w_2)^\ell(\bar z_1w_1+\bar z_2w_2)^\ell\ .
$$
It is a polynomial in $z$ if and only if $d\ge \ell$. 

As for ii), if $P$ is a polynomial, then $P$ is $\spaziW_n^\ell$-valued if and only if $P(\base)$ is in $\spaziW_n^\ell$. 
Therefore $P(\base)$ is a constant multiple of $B_n^\ell$.
\end{proof}

Lemma~\ref{l:estensione-matricediag} gives a recipe to find a new basis for $\big(\cP(\bC^2)\otimes{\rm End}(\spaziorp_n)\big)^K$
and
at the same time, proves that the operators $\Delta_z$ and $\DD_n$ generate the algebra
$\big(\mathbb{D}(\bC^2)\otimes \textup{End}(\spaziorp_n)\big)^K$, independently of~\cite{FRY1}.
We state
 these facts in Corollary~\ref{c:gen}, where we use the following notation. 
Let $\widehat F$ be the Fourier transform of  $F\in L^1\big(\bC^2,\textup{End}(\spaziorp_n)\big)^K$, defined component-wise by
\[
\widehat F(\zeta)=\int_{\bC^2} F(z)\, e^{-i\langle z,\zeta\rangle}\, dz
\]
and let $P$ be an $\textup{End}(\spaziorp_n)$-valued polynomial. Then $P(\partial)$ is the operator defined by the rule
\[
\widehat{P(\partial)F}(\zeta)=P(\zeta)\widehat F(\zeta)\qquad \zeta\in \bC^2.
\]
 In particular, if $P(z)=z^\alpha\, \bar z^\beta I$, then $P(\partial)= (-2i\partial_{\bar z})^\alpha\,(-2i\partial_{ z})^\beta I$.

\begin{corollary}\label{c:gen}

 \begin{itemize}
\item[\rm(i)] 
The polynomials   $Q^k_{B^\ell_n}$, where $\ell=0,1,\dots,n$  and $k\ge \ell$, form a basis of
 $\big(\cP(\bC^2)\otimes{\rm End}(\spaziorp_n)\big)^K$ as $\cP(\bC^2)^K$-module.
 \item[\rm(ii)] A set of generators of the algebra $\big(\mathbb{D}(\bC^2)\otimes \textup{End}(\spaziorp_n)\big)^K$
is  
$$
\cD=\{\Delta_z=Q^1_{B^0_n}(\partial),\quad\DD_n=Q^1_{B^1_n}(\partial)\}.
$$
\end{itemize}
\end{corollary}
\medskip

\subsection{${\rm End}(\spaziorp_n)$-valued spherical functions as derivatives of scalar spherical functions} 
\quad
\medskip

For $\xi\ge0$ we denote by $\varphi_\xi$ the  spherical function of $(G,K)$ as  a Gelfand pair with eigenvalue $\xi$ relative to $\Delta_z$. In formulae,
\be\label{Bessel}
\varphi_\xi(z)=\varphi_1(\sqrt\xi z)\quad\text{ where }\quad \varphi_1(z)=\int_{S^3}e^{-i\lan z,\zeta\ran}\, d\sigma (\zeta)\ \ \  \Big(=\frac{1}{|z|/2}\,J_{1}\big(|z|\big)\Big)
\ee
and $\sigma$ is the normalized surface measure of the unit sphere $S^3$ in $\bC^2$.

\begin{proposition}\label{spettro}
The spectrum $\Sigma^n_\cD$ is the union of $n+1$ half-lines,
$$
\Sigma^n_\cD=\big\{\big(\xi,(-n+2j)\xi\big):\xi\ge0\,,\ j=0,\dots,n\big\}\ .
$$
If $\xi=0$, the only pair of eigenvalues $(0,0)$ is attained by the constant spherical function $\Phi_{0,0}(z)=I$.
\\
For $\xi>0$, the matrix-valued spherical function corresponding to the pair of eigenvalues
 $\big(\xi,(-n+2j)\xi\big)$ is 
\bea\label{e:Phixij}
\Phi_{\xi,j}(z)&=(n+1) \int_{S^3}e^{-i\sqrt\xi\lan z,\zeta \ran}Q_{E_{jj}}^n(\zeta)\,d\sigma(\zeta)
\\
&=\xi^{-n}\,(n+1)Q^n_{E_{jj}}(\de)\ph_\xi( z)
\\
&=(n+1)\big(Q^n_{E_{jj}}(\de)\ph_1\big)\big(\sqrt\xi z\big),
\eea
where $E_{jj}$ is the matrix with null entries except the $jj$-entry which equals $1$.
 \end{proposition}

\begin{proof}
By \cite[(11.2) and Thm. 11.1]{RS}, the bounded spherical function of the triple $(G,K,\tau_{n,n})$ can be constructed according to the following recipe.

Fix $r\base$, with $r=\sqrt\xi\ge0$, as base point in the $K$-orbit $r S^3$ in $\bC^2$, let $K_r$ be the stabilizer of $r\base$ in $K$, decompose $\spaziorp_n$ into its (inequivalent) irreducible components $W_{r,j}$ under $K_r$, and define
$$
\Phi_{\xi,j}(z)=\frac{n+1}{\dim W_{r,j}}\int_Ke^{-ir\lan kz,\base\ran}\tau_{n,n}(k\inv)P_{r,j}\tau_{n,n}(k)\,dk\ ,
$$
where $P_{r,j}$ is the orthogonal projection onto $W_{r,j}$.

If $r=0$, $K_r=K$ and we obtain constant function, $\Phi_0(z)=I$.

If $r>0$, $K_r$ is the torus of diagonal matrices $\begin{pmatrix}e^{i\theta}&0\\0&1\end{pmatrix}$, so each $W_{r,j}$ is the one-dimensional span of $e^j_n$ in \eqref{basisofVn}, for $j=0,\dots,n$.

Consequently, $P_{r,j}$ is represented by the matrix $E_{jj}$ in the basis $\{e^j_n\}$ and, by \eqref{P->B}
 and~\eqref{Bessel} for $r=\sqrt{\xi}$ we obtain
\begin{align*}
%\label{formulaPhi}
\Phi_{\xi,j}(z)&= (n+1) \int_Ke^{-ir\lan kz,\base\ran}\tau_{n,n}(k\inv)E_{jj}\tau_{n,n}(k)\,dk\\
    &=(n+1) \int_Ke^{-ir\lan z,k\inv\base \ran}Q_{E_{jj}}^n(k\inv\base)\,dk  \\
    &=(n+1) \int_{S^3}e^{-ir\lan z,\zeta \ran}Q_{E_{jj}}^n(\zeta)\,d\sigma(\zeta)\\
    &=(n+1)\big(Q_{E_{jj}}^n(\de)\varphi_1\big)(\sqrt{\xi}z)
    \\
    &=(n+1)\, \xi^{-n}\, \big(Q_{E_{jj}}^n(\de)\varphi_\xi\big)(z)\ .
\end{align*}

This proves that the set of $n+1$ functions in \eqref{e:Phixij} is the set of spherical functions whose eigenvalue relative to $\Delta_z$ is $\xi$. It remains to determine the eigenvalue relative to $\DD_n$ for each of them. Taking into account that $B^1_n E_{jj}=(-n+2j)E_{jj}$, we have
\begin{align*}
\DD_n\Phi_{\xi,j}&=(n+1)\,\xi^{-n}\,Q_{B^1_n}^1(\de)\,Q_{E_{jj}}^n(\de)\varphi_\xi \\
&=(n+1)\,\xi^{-n}\,Q_{B^1_n E_{jj}}^{n+1}(\de)\varphi_\xi\\
&=(n+1)\,\xi^{-n}\,(-n+2j)\, Q_{E_{jj}}^{n+1}(\de)\varphi_\xi\\
&=(n+1)\,\xi^{-n}\,(-n+2j)\, Q_{E_{jj}}^{n}(\de)\Delta_z\varphi_\xi \\
&=(-n+2j)\xi\, (n+1)\,\xi^{-n}\, Q_{E_{jj}}^{n}(\de)\varphi_\xi\\
&=(-n+2j)\xi\,\Phi_{\xi,j}\ . \qedhere
\end{align*}
\end{proof}

\medskip

\subsection{${\rm End}(\spaziorp_n)$--valued equivariant functions as derivatives of scalar valued functions}
\quad\\
\medskip
Suppose that $F$ is an ${\rm End}(\spaziorp_n)$--valued equivariant function. Then we can decompose $F$ into the sum
\[
F=\sum_{\ell=0}^nF_\ell
\]
where each $F_\ell$ is $\spaziW_n^\ell$--valued. We are going to prove that $\spaziW_n^\ell$-valued functions turn out to be of a special form. 

For our purposes it is convenient to  consider on $\Smatr$  the following  family of  norms  
\be\label{M-norms}
\|F\|_{(M)}=\max_{0\leq q\leq M}\|(1+|\cdot|^2)^M\, \Delta_z^q F\|_2,
\ee 
where the $L^2$ norm of $\textup{End}(\spaziorp_n)$--valued functions is defined in \eqref{norma2}.

\medskip
For the sake of brevity, we denote by $Q^\ell_n$ the polynomial $Q^\ell_{B^\ell_n}$.

\begin{proposition}\label{sommaF} 
Let $F$ be 
in $\cS(\bC^2,\spaziW_n^j)^K$.
Then there exists a scalar valued invariant Schwartz function $g_0$
 such that
\[
 F=Q_n^\ell (\de)g_0=q_n^\ell(\mathbf D_n)g_0\ ,
\]
and $g_0$ is of the form  $g_0=g(|\cdot|^2)$, where $g$ is an even Schwartz function  on the real line.
Moreover, for any $M$  there exists $M'\geq M+\ell$ such that
\[
\|g\|_{(M)}\leq C_M\, \|F\|_{(M')}\ .
\]
\end{proposition}

\begin{proof} By  equivariance, it follows that
\[
\begin{aligned}
&F(0,e^{i\theta}r)=F(e^{i\theta}\cdot (0,r))=F(0,r)
\\
&F(0,r)=F(\exp(\theta \baseuno)\cdot (0,r))=\tau_n(\exp(\theta \baseuno))\,F(0,r)\, \tau_n(\exp(-\theta \baseuno))
\end{aligned}
\qquad \forall r,\theta\in \bR
\]
so that the matrix $F(0,r)$ is diagonal. As  $F(0,r)$ is in $\spaziW_n^\ell$, by Lemma~\ref{diagonalWnj}(i)
we can write $F(0,r)=f(r)\,B_n^\ell$ for some scalar $f(r)$. Clearly the so-obtained function $f$ is in $\cS(\bR)$ and
even.

Suppose that $P_{d}$ is the homogeneous term of degree $d$ in the Taylor expansion of $F$ centred at the origin.
Then $P_d$ is $K$-equivariant and $\spaziW_n^\ell$-valued. It follows from Lemma \ref{l:estensione-matricediag} (i) that $P_d=0$ if $d<2\ell$
and $P_{2\ell}$ is a constant multiple of $Q_{n}^\ell$. Hence
\[
f(r)=c_\ell\, r^{2\ell}+o(r^{2\ell}) \qquad r\to 0\ .
\] 
By Hadamard's division Lemma~\cite{Jet}
there exists an even smooth function $h$ on $\bR$ such that $f(r)=r^{2\ell}\, h(r)$ and $h(0)=c_\ell$,
therefore for $z=|z|z'$
\[
F(z)= h (|z|)\, |z|^{2\ell}\,\tau_n(k_{z'})\,B_n^\ell\tau_n(k_{z'})^*
= h_0(z)\,Q^\ell_n(z) 
\]
where $h_0$ is a scalar invariant Schwartz function on $\bC^2$ and for any $M'$
\[
\|h_0\|_{(M')}\leq C_{M'}\, \|h\|_{(M')}\leq C_{M'}\, \|f\|_{(M'+\ell)}\leq C_{M'}\, \|F\|_{(M'+\ell)}.
\]
Since the Fourier transform commutes with the action of $\grpU$, the same kind of result holds for the Fourier transform of $F$.
Therefore there exists an invariant Schwartz function 
$\gamma_0$ on $\bC^2$ such that
\[
\hat F(\zeta)=Q^\ell_n (\zeta)\, \gamma_0(\zeta)
\]
and taking the inverse Fourier transform
\[
F(z)=Q^\ell_{n} (\partial)\, g_0(z),
\]
with
\[
\|g_0\|_{(M')}\leq C_{M'}\,\|\gamma_0\|_{(M')}\leq C_{M'}\,\|\hat F\|_{(M'+\ell)}\leq C_{M'}\,\| F\|_{(M'+\ell)}.
\]
By~\cite{GS} (see also \cite{Mather}), the function $g_0$ is of the  form $g_0=g(|\cdot|^2)$ with $g\in \cS(\bR)$
and, for any given $M$, there exists  $M'\geq M$ such that
\[
\|g\|_{(M)}\leq C_M\, \|g_0\|_{(M')}.
\qedhere
\]
\end{proof}
\bigskip

\section{Schwartz correspondence for $\Smatr$ }\label{sect-Gn}
According to~\eqref{Gn} , we denote by $\cGn F$ the $\tau_n$-spherical transform   of  a function $F$ in $ \Lmatr$  given  by 
$$
\cGn F(\xi,\xi(-n+2j))= \frac1{n+1}\, \int_{\bC^2}\tr\big(F(z)\Phi_{\xi,j}(-z)\big)\,dz \qquad \forall \xi\geq 0, \quad j=0,1,\ldots,n\ .
$$

The next two subsections will provide the proof of the following theorem.

\begin{theorem}\label{(S)ridotto} 
The map $\cGn$ is a 
 isomorphism of 
 $\Smatr$  onto $\cS(\Sigma_\cD^n)$.
\end{theorem}

\subsection{Schwartz extensions of  $\cGn F$}
\quad
 
 \medskip

 We begin by proving  that  $ \cGn $ maps $\Smatr $ into $\cS(\Sigma_\cD^n)$ and that it is continuous.
For $j=0,1,\ldots,n$ denote $t_j=-n+2j$.

\begin{lemma}\label{FourierF}
Let $F$ be in 
$\Smatr$. Then the following hold.
\begin{enumerate}[\rm(i)]
\item 
$\displaystyle{\widehat F( \sqrt\xi\,\base)\,e^j_n=\cGn F\big(\xi, t_j\xi\big)\,e^j_n
 \qquad \forall \xi\geq 0, \quad j=0,1,\ldots,n.
 }
$

\medskip

\item There exist  $\gamma_0,\ldots, \gamma_n$ in $  \cS(\bR)$ such that 
$$
\cGn F\big(\xi, t_j\xi\big)=
\sum_{\ell=0}^n \xi^\ell\,q^\ell_n ( t_j)\,\gamma_\ell(\xi)
\qquad \forall \xi\geq 0, \quad j=0,1,\ldots,n,
$$
where $q^0_n, \ldots q^n_n$ are the  polynomials defined in~{\rm\pref{e:matriciBj}}.
\end{enumerate}

\end{lemma}
 
 \begin{proof}

 Let $F$ be in 
$\Smatr$. Using the equalities   in~\eqref{e:Phixij}, we obtain
\beas
\cGn F(\xi,\xi t_j)
&= \frac1{n+1}\, \int_{\bC^2}\tr\big(F(z)\Phi_{\xi,j}(-z) \big)\,dz
\\
   &=  \int_{\bC^2} \int_{S^3}e^{-i\sqrt \xi\lan z,\zeta \ran}\tr\big(F(z)Q_{E_{jj}}^n(\zeta)\big)\,d\sigma(\zeta)\,dz
   \\
&=
\int_{S^3}\tr\big(\widehat F(\sqrt \xi\,\zeta) Q^n_{E_{jj}}(\zeta)  \big)\,d \sigma (\zeta)
\\
&=
\tr\big(\widehat F(  \sqrt\xi\base )\,E_{jj} \big)
  \\
&=
   \big(\widehat F( \sqrt\xi\base )\big)_{jj}\ .
\eeas
In order to prove (ii), we  decompose $F $ into the sum
$\displaystyle
F=\sum_{\ell=0}^nF_\ell
$
where each $F_\ell$ is in  $\SmatrW$.
By Proposition~\ref{sommaF}, for each $\ell$ there exists a  function $\gamma_\ell \in \cS(\bR)$ such that 
 $$
\widehat F_\ell(  \sqrt\xi\,\base)= \xi^\ell \, \gamma_\ell (\xi) \,B^\ell_n= \xi^\ell \, \gamma_\ell (\xi) \, \mathrm{diag}\left(q^\ell_n(t_0),\ldots,q^\ell_n( t_n) \right)
\qquad \forall \xi\geq 0
 $$
 and (ii) follows.
\end{proof}

\begin{corollary}\label{estensioneGn} 
Let $F$ be in 
$\Smatr$ . Then there exists $g\in\cS(\bR^2)$ such that 
${g}_{ |_{\Sigma_\cD^n}}=\cGn F$. Moreover, for every $M$ there exist $M'>M+n$ and a constant $C_{M,n}$ such that 
\[
%\label{contGn}
\|g\|_{(M)} \leq C_{M,n} \| F\|_{(M')}.
\]
\end{corollary}

\begin{proof} Notice that   $\Sigma_\cD^n$ is contained in
 $C_n=\left\{(\xi_1, \xi_2)\in \bR^2\,\,:\,\, |\xi_2|\leq n \xi_1, \quad \xi_1\geq 0 \right\}$ .
Let $\eta$ be a smooth function on $\bR^2$ with bounded derivatives  of any order which takes value 1 on
 $C_n $ and vanishes outside  $C_n-(\varepsilon,0)$, for some $\varepsilon>0$.  
 Let $F$ be in $\Smatr$ and let $\gamma_0,\ldots, \gamma_n$ in $  \cS(\bR)$ be as in Lemma~\ref{FourierF} such that 
\[
\begin{aligned}
\cGn F\big(\xi, ( -n+2j)\xi\big)&=
\sum_{\ell=0}^n
 \sum_{k=0}^\ell a_{k,\ell}\,\big( \xi(-n+2j)\big)^k\,\xi^{\ell-k} \,\gamma_\ell(\xi)
\qquad \forall \xi\geq 0, \quad j=0,1,\ldots,n.
\end{aligned}
\]
where $a_{k,\ell}\in \bC$ are the coefficients of the polynomial $q^\ell_n$, i.e.
$q^\ell_n(t)=\sum_{k=0}^\ell a_{k,\ell}\, t^k$.
Then the  function $g$ defined on $\bR^2$ by 
$$
g(\xi_1, \xi_2)=\eta_n(\xi_1, \xi_2) \, \sum_{\ell=0}^n \sum_{k=0}^\ell a_{k,\ell}\, \xi_2^k\,\xi_1^{\ell-k} \,\gamma_\ell(\xi_1)
\qquad \forall (\xi_1, \xi_2)\in \bR,
$$
satisfies  the required properties. 
\end{proof}

\subsection{Surjectivity of $\cGn$ and Schwartz  correspondence for $\Smatr$
}
 \quad 
 \medskip
 
 We conclude the proof  of Theorem~\ref{(S)ridotto} by proving  that   the continuous linear map $\cGn\,\,:\,\,  \Smatr \longrightarrow \cS(\Sigma_\cD^n)$ is surjective.

\begin{proposition}\label{onto}
Let $g$ be in $\cS(\bR^2)$. Then   there exists $F$ in $\Smatr$ such that
$\cGn F=g_{|_{\Sigma_\cD^n}}$.
\end{proposition}

\begin{proof} For  $ j=0,1,\ldots,n$ let $t_j=(-n+2j) $ and
fix $\xi>0$ . Denote by $p_\xi$ 
the polynomial such that $p_\xi(t_j)=g\big(\xi ,\xi t_j\big)$, $ j=0,1,\ldots,n.$. Using Newton's interpolation formula, we write
$$
p_\xi (t) =
\mu_0(\xi )+
\xi\, \mu_1(\xi )(t-t_0)
+\xi^2\mu_2(\xi )(t-t_0)(t-t_1)
+\cdots+ \xi^n\mu_n(\xi )(t-t_0)\cdots(t-t_{n-1}).
$$
Then by the Hermite--Genocchi formula~\cite{Atk}, which we express in the equivalent form for equidistant points, we have
$$
\mu_\ell(\xi)=\frac1{\ell!}\int_0^1\int_0^{1}\cdots \int_0^{1}\partial_2^{(\ell)}g(\xi,\xi(-n+2u_1+\cdots+2u_\ell))\,du_{\ell}\cdots du_2\,du_1
\quad  \forall \xi>0,
$$
$\ell=0,1,\ldots,n.$. 

Via this formula we extend $\mu_0,\mu_1\ldots,\mu_n$ to Schwartz functions on $\bR$ and we can write  
$$%\label{muj}
g(\xi, \xi t_j)=\sum_{0\le k\le\ell\le n}  b_{k,\ell}\,(t_j\xi)^k\xi^{\ell-k}\,\mu_\ell(\xi ),
\qquad \forall \xi\geq 0, \quad j=0,1,\ldots,n,
$$
for some complex numbers $b_{k,\ell}$, where $0\leq k\leq \ell\leq n$.

Define 
$$
f_\ell(z)= \frac1{(2\pi)^4} \int_{\bC^2} \mu_\ell (|\zeta|^2)\, e^{i\lan z,\zeta\ran}\, d\zeta,
$$
then $\mu_\ell=\cG_0f_\ell$, $\ell=0,1,\ldots,n$ and the
function  
$$
F=\sum_{0\le k\le\ell\le n} b_{k,\ell} {\mathbf D}_n^k\Delta_z^{\ell-k}f_\ell
$$
satisfies the required properties.
\end{proof}

\begin{proof}[Proof of Theorem~\ref{(S)ridotto} ]   By 
 Corollary~\ref{estensioneGn} $\cGn$ maps $\Smatr$ into  $\cS(\Sigma_n)$  continuously and, by Proposition~\ref{onto}, it is surjective. It follows from the open mapping theorem for Fr\'echet spaces \cite{Tr} that also $\cG_n\inv$ is continuous.
  \end{proof}

\subsection{Norm estimates for $\cGn\inv$ with polynomial growth}
\quad
\medskip

Theorem~\ref{(S)ridotto} will be required in the proof of Schwartz correspondence for the strong Gelfand pair $(U_2\ltimes\bC^2,U_2)$. However, something more will be needed, i.e., that the dependence on $n$ of the Schwartz norm estimates for $\cG_n\inv$ is polynomial.

This fact can be deduced from a general result in \cite[Prop. 4.2.1]{Mar1} for weighted subcoercive systems of left-invariant differential operators on Lie groups with polynomial growth. However, we give an independent and relatively simple proof, well adapted to our case.

At this stage we disregard the issue of polynomial growth of the Schwartz norm estimates for the direct spherical transforms $\cGn$, because they are not needed in the proof. They will follow however once property (S) for the strong Gelfand pair $(G,K)$ will be established, see Corollary~\ref{Gnpolinomiale}.  Identifying  the  unit sphere  $S^3$  with the group $\grpSU$ as in~\eqref{e:id-sfera-su2},
 the expression of the  Laplacian in polar coordinates  takes the form
\bea\label{polari}
\Delta_z=-\partial_r^2-\frac{3}{r}\partial_r+\frac{1}{r^2}\Cas,
\eea
where $\Cas=-X_1^2-X_2^2-X_3^3$ is the Casimir operator  on  $\grpSU$.
In the   next lemma we determine the action of the laplacian on smooth    ${\rm End}(V_n)$--valued  functions.

\begin{lemma}
\label{azioneDelta}
Let  $F$ be in $\Smatr$
with $ F(0,\cdot)=\mathrm{diag}(f_0,\ldots f_n)$, then
\[
%\label{casimir}
(\Cas F)
(r\base)e_n^j 
= \big( (t_j^2-n^2-2n)(
f_{j+1}-2f_j+f_{j-1}
 )+2t_j (
f_{j+1}-f_{j-1}
 )  
\big)(r)e_n^j 
\]
 where $t_j=-n+2j$, $j=0,1,\ldots,n$.
\end{lemma}

\begin{proof} Let  $F$ be in $ \cS\big(\bC^2,{\rm End}(\spaziorp_n)\big)^K$.
Because of the $\tau_n$--invariance, we have 
 \[
\begin{aligned}
\Cas \Big(\tau_n(k)F(r\base )\tau_n(k)^\ast \Big)=&
\tau_n(k)\,\,C F
(r\base)\,\,
\tau_n(k)^\ast
\qquad \forall r>0,\quad \forall k\in \grpSU\ .
\end{aligned}
\]
 Notice that 
\[
\begin{aligned}
\Cas=-X_1^2-X_2^2-X_3^3
&=-X_1^2-2iX_1-(X_2+iX_3)(X_2-iX_3),
\\
&=-X_1^2+2iX_1-(X_2-iX_3)(X_2+iX_3)
\end{aligned}
\]
and that, for every   $X\in \fs\fu_2$, we have 
$$
X\tau_n(k)=\tau_n(k)\, d\tau_n(X)\qquad \text{and}\qquad X\tau_n^\ast(k)= - d\tau_n(X)\, \tau_n^\ast(k).
$$
 Hence we obtain
$$
\Cas\tau_n(k)=\tau_n(k)\, d\tau_n(\Cas),\qquad 
\Cas\tau_n(k)^\ast=d\tau_n(\Cas)\tau_n^\ast(k).
$$
Moreover $d\tau_n(\Cas)$ and   $id\tau_n(X_1)=B_n^1$ commute with $\Cas F
(r\base)$ so that 
$$
(\Cas\tau_n(k))\,\, F
(r\base)\,\,
  \tau_n(k)^\ast+\tau_n(k)\,\, F
(r\base)\,\,
(\Cas \tau_n(k)^\ast)
=
2\,\tau_n(k)\,d\tau_n(\Cas)\, F(r\base)\,
  \tau_n(k)^\ast
$$  
and
$$
(X_1\tau_n(k))\,\, F
(r\base)\,\,
(X_1 \tau_n(k)^\ast)=
\tau_n(k)(B_n^1)^2\, F
(r\base)\tau_n(k)^\ast.
$$
Therefore by   the Leibnitz rule  
$$
(\Cas F)(r\base )=
2\left(d\tau_n(\Cas)-(B_n^1)^2\right) F(r\base )
+\Lambda F(r\base )\Lambda^\ast
 +\Lambda^\ast F(r\base )\Lambda
$$
where
$$
\Lambda =d\tau_n(X_2+iX_3) \quad {\text{ so that }} \quad  \Lambda^\ast=-d\tau_n(X_2-iX_3).
$$
  
Recalling that $d\tau_n(\Cas)=(n^2+2n)I$ and $B_n^1e_n^j=t_j\,e_n^j$, we only need to evaluate  $\Lambda F(r\base)\Lambda^\ast
 +\Lambda^\ast F(r\base)\Lambda$.
 Since $\Lambda  e_n^j =2\sqrt{(j+1)(n-j)}\, e_n^{j +1 }$ and $\Lambda^\ast e_n^j=  2\sqrt{\ell(n-j+1)}\, e_n^{j -1 }$, then 
\[
\begin{aligned}
\big(
 \Lambda F(r\base)\Lambda^\ast 
 +\Lambda^\ast F(r\base)\Lambda\, \big)e_n^j&=-4j(n-j+1)f_{j-1}(r)e_n^j-4(j+1)(n-j)f_{j+1}(r)e_n^j
 \\
 &=\Big((t_j^2-2t_j-n^2-2n\big)f_{j-1}+
\big(t_j^2+2t_j-n^2-2n\big)f_{j+1}\Big)(r)
e_n^j.
\end{aligned}
\]
  \end{proof}
  
  \begin{proposition}\label{inversocontinuo}
  For every $M\in \bN$ there exists $N_M\in\bN$ such that,
for every $n$ and every  $g$   in $\cS(\bR^2)$,  
$$
\|\cGn^{-1}(g_{|_{\Sigma_\cD^n}})\|_{(M)}\leq  C_{M,n}\|g\|_{(N_M)},
$$ 
where the constant $C_{M,n}$ has polynomial growth in $n$.
\end{proposition}

\begin{proof}  
Let $g$ be in $\cS(\bR^2)$. By Proposition~\ref{onto} and  Lemma~\ref{FourierF} we know that $\cGn$ is bijective and that  the  function    $F=\cGn^{-1}(g_{|_{\Sigma_\cD^n}})\in  \Smatr$ satisfies the equality
$$
 \widehat F(\sqrt\xi\,\base)e^j_n=g\big(\xi, (-n+2j)\xi\big)e^j_n
 \qquad \forall \xi\geq 0\quad 0\leq j\leq n\ .
 $$

  We have
   \[
 \begin{aligned}
 \|(1+|\cdot|^2)^M\, \Delta_z^q \widehat F\|_{2}^2
&=
 \int_{\bC^2}(1+|\zeta|^2)^{2M}\,\left\| \Delta_z^q \widehat F(\zeta) \right\|_{HS}^2\, d\zeta
 \\
 &=|S^3|\,  \int_0^{+\infty}\int_K\,
 \left\| \Delta_z^q \widehat F(k\cdot r\base) \right\|_{HS}^2 dk\,(1+r^2)^{2M}\,r^3\,dr 
 \\
  &=|S^3| \, \int_0^{+\infty}(1+r^2)^{2M}\,\sum_{j=0}^{n}\left|\left(\Delta_z^q \widehat F(r\base)\right)_{jj}\right|^2\,r^3\,dr .
 \end{aligned}
 \]
 
 We now  compute $\Delta_z \widehat F(r\base)$ using the polar decomposition~\eqref{polari}  and Lemma~\ref{azioneDelta}.
 The action of $\partial_r^2+\frac3{r} \partial_r$ 
 on  the function  
 $$
 r\longmapsto  \widehat F(r\base) = {\rm diag}\Big(g(r^2, r^2t_0),\ldots g(r^2, r^2t_n) \Big)
 $$
   is given by
 \beas
 \left(\partial_r^2+\frac3{r} \partial_r\right)g(r^2, r^2t_j)
=
\left(8(\partial_1+t_j\partial_2)g+4r^2(\partial_1+t_j\partial_2)^2g \right)(r^2, r^2t_j) .
 \eeas 
In order to compute the action 
 of the Casimir operator $\Cas$,
 we apply  formula~\eqref{azioneDelta} with $f_j(r)=g(r^2, r^2t_j)$. 
The Taylor expansion in the second variable of $g$ gives, with $\xi=r^2$,
 \beas
 f_{j\pm k}(\sqrt \xi)&=g(\xi,\xi(t_j\pm 2 k))
 \\
 &=\sum_{s=0}^{p} \frac{\partial_2^s g\left(\xi, \xi t_j\right)}{s!}
\, (\pm2k\xi)^s+(\pm2k\xi)^{p+1}\int_0^1 \frac{\partial_2^{p+1} g\left(\xi, \xi ( t_j\pm2ku) \right)}{p!}\, (1-u)^p\, du.
  \eeas
So that 
  \beas
 (f_{j+1}-2f_j+f_{j-1})(\sqrt \xi)&
 =g(\xi,  \xi t_j+2\xi)-g(\xi, \xi t_j)+g(\xi, \xi t_j-2\xi)
 \\
 &=
(2\xi)^2 \int_0^1
\left( \partial_2^2 g\left(\xi, \xi(t_j+2u)\right)+\partial_2^2 g\left(\xi, \xi(t_j-2u)\right)
\right)\,(1-u) \,du
 \eeas
 and 
  \beas
 (f_{j+1}- f_{j-1})(\sqrt \xi)&
 =g(\xi,  \xi t_j+2\xi)- g(\xi, \xi t_j-2\xi)
 \\
 &=
2\xi \int_0^1
\left( \partial_2 g\left(\xi, \xi(t_j+2u)\right)+\partial_2 g\left(\xi, \xi(t_j-2u)\right)
\right)\, du.
 \eeas

  Therefore, by \eqref{polari} and Lemma~\ref{azioneDelta},
 
 \beas
-\big(\Delta_z \widehat F&(r\base)\big)_{jj}
= \left(\partial_r^2+\frac3{r} \partial_r\right)(g(r^2, r^2t_j))-\frac1{r^2}(\Cas \widehat F)(r\base)
\\
=&
\left(4r^2(\partial_1+t_j\partial_2)^2g+8(\partial_1+t_j\partial_2)g \right)(r^2, r^2t_j) 
\\
&-
4 r^2(t_j^2-n^2-2n) \int_0^1\left( \partial_2^2 g\left(r^2, r^2(t_j+2u)\right)+\partial_2^2 g\left(r^2, r^2(t_j-2u)\right)
\right) (1-u)\,du
\\&
-4t_j \int_0^1
\left( \partial_2 g\left(r^2, r^2(t_j+2u)\right)+\partial_2 g\left(r^2, r^2(t_j-2u)\right)
\right)\, du.
 \eeas

Since $|t_j|\leq n$, $j=0,1,\ldots ,n $, by iteration, we obtain
 $$
\left|\big(\Delta_z^q \widehat F(r\base)\big)_{jj}\right| 
\leq C_q\, (1+r^{2q})\, n^{2q}\, \sum_{s=1}^{2q}\sup_{u\in \bR}\left|\partial^s g(r^2,r^2u)\right|.
 $$
 Therefore 
 \[
  \|(1+|\cdot|^2)^M\, \Delta_z^q \widehat F\|_{2}
   \leq C_{q,M}\, n^{2q+1}\, \max_{0\leq | \beta|\leq 2q}
\|(1+|\cdot |^2)^{M'}\,\partial^\beta g\|_{\infty} 
 \]
 for some $M'>M$.
  Therefore for every $M\in\bN$ there exists $M_n\in \bN$ such that 
 \[
 \|F\|_{(M)}\leq C_M \|\widehat F\|_{(M)}\leq  C_{M,n}\|g\|_{(N_M)}.
 \qedhere
 \]
 \end{proof}
 
\section{Schwartz correspondence for $\cS\big(M_2(\bC)\big)^{{\rm Int}(U_2)}$}

In this section we prove property (S') of Theorem~\ref{rendiconti},  which implies  the Schwartz correspondence for the strong Gelfand pair $(G,K)$. 

We shall deal with the Schwartz norms of a function $f$ of a given type $(m,n)$ and of the corresponding matrix valued function~$A_{m,n}f$. Here we quantify the relation between these norms.

As $M$-order Schwartz norm of a function $f$ in $\cS(G)^{{\rm Int}(K)}$ 
we take
\[
\|f\|_{(M)}=\max_{q,r,s=0\ldots M}\|(1+|z|^2)^M\,D_4^sD_3^rD_1^q f\|_2
\] 
where $D_1=\Delta_z$, $D_3=\Cas$ and $D_4=i\basequattro$.

\begin{lemma}\label{norme-f-Af} Let $(m,n)$ be in $E$. The following estimates hold 
\[
\|A_{m,n}f\|_{(M)}\leq \frac{1}{ \sqrt{n+1}}\, \|f\|_{(M)} \qquad \forall f\in \cS(G)^{{\rm Int}(K)}_{\tau_{m,n}}
\]
and conversely
\[
\|A_{m,n}^{-1}F\|_{(M)}\leq (1+|m|)^M\,(1+n)^{2M+1/2}\, \|F\|_{(M)} \qquad \forall F\in \Smatr\ .
\]
\end{lemma}

\begin{proof} Note that when $f$ is of type $(m,n)$,
\[
D_3f=(n^2+2n)f\qquad D_4f=mf.
\]
The estimates follow easily from the fact that $\sqrt{n+1}A_{m,n}$ is an isometry on the corresponding $L^2$-spaces.
\end{proof}

By \eqref{dec-sigma}, $\Sigma_\cD$ decomposes as the union of
$$
\Sigma_\cD^{m,n}=\Sigma_\cD^n\times\big\{(n^2+2n,m)\big\}\ ,\qquad (m,n)\in E\ .
$$

At this stage we abandon the ${\rm End}(V_n)$-valued picture, and reinterpret 
Corollary~\ref{estensioneGn} and Propositions \ref{onto}, \ref{inversocontinuo} in the following form, 
using the fact that $\cG_{\tau_{m,n}}=\cG_n\circ A_{m,n}$.

\begin{corollary}\label{7.3+7.6}
\quad

\begin{enumerate}[\rm(i)]
\item Given $\tau_{m,n}\in\widehat K$ and $f\in\cS(G)^{{\rm Int}(K)}_{\tau_{m,n}}$, the spherical transform $\cG f$,
which is supported on $\Sigma_\cD^{n}\times\big\{(n^2+2n,m)\big\}$, admits a Schwartz extension to $\bR^2\times\big\{(n^2+2n,m)\big\}$, and hence a Schwartz extension to $\bR^4$ which vanishes on the other components of $\Sigma_\cD$.
\item For every $(m,n)\in E$, the transform $\cG_{\tau_{m,n}}$ is an isomorphism from $\cS(G)^{{\rm Int}(K)}_{\tau_{m,n}}$ to $\cS(\Sigma_\cD^n)$.
\item  For every $M\in \bN$ there exists $N_M\in\bN$ such that,
for every $(m,n)\in E$ and every  $g$ in~$\cS(\bR^2)$,  
$$
\|\cG_{\tau_{m,n}}^{-1}(g_{|_{\Sigma_\cD^n}})\|_{(M)}\leq  C_{M,m,n}\|g\|_{(N_M)},
$$ 
where the constants $C_{M,m,n}$ have polynomial growth in $(m,n)$. 
\end{enumerate}

\end{corollary}

If we consider now a general $f\in\cS(G)^{{\rm Int}(K)}$, 
$$
f=\sum_{(m,n)\in E}f_{m,n}\ ,
$$ 
we cannot prove, on the basis of the results in Section \ref{sect-Gn}, that the Schwartz extensions to $\bR^4$ constructed in the proof of Theorem~\ref{rendiconti} for the individual $\cG_\cD f_{m,n}$ add up to give a Schwartz function. 

In order to do so, we need to proceed to a new construction of Schwartz extensions, possibly different from those already available, which gives, for any finite number of Schwartz norms, rapid decay as $n$ goes to infinity (rapid decay in $m$ for fixed $n$ is trivial). 

 It is interesting to notice that this new construction does not replace the work done in Section 7 because it requires to know in advance that a Schwartz extension  whatsoever exists for each $(m,n)$.

\bigskip
\subsection{Jets with polynomial growth for each $K$-type}
\quad
\medskip

For  $f\in \cS(G)^{{\rm Int}(K)}_{\tau_{m,n}}$, we have a simple estimate on the directional derivatives of $\cG_{\tau_{m,n}}f$  
 in the $n+1$ directions of the half-lines forming $\Sigma^n_\cD$.
 
\begin{lemma}\label{Gderivatives}
For  $d\in \bN$,  there exists constants $C_d$ and $N_d$ independent of $(m,n)$ such that
$$
\left|(d/d\xi_1)^d\cG_{\tau_{m,n}}f\big(\xi_1,(-n+2j)\xi_1\big)\right| \le C_{d} \,\|f\|_{(N_d)}\ ,
\qquad \forall f\in \cS(G)^{{\rm Int}(K)}_{\tau_{m,n}}
$$
for all $\xi_1\ge0$.
\end{lemma} 

\begin{proof} 
From \eqref{matrix-sph} it follows that
\[
\cG_{\tau_{m,n}}f\big(\xi_1,(-n+2j)\xi_1\big)=\cG_n(A_{m,n}f)\big(\xi_1,(-n+2j)\xi_1\big).
\]
Since $\Phi_{\xi_1,j}$ is even in $z$, we obtain from \eqref{e:Phixij}
that, for $\xi_1>0$,
$$
\Phi_{\xi_1,j}(z)=(n+1) \int_{S^3}\cos(\sqrt\xi\lan z,\zeta \ran)\,Q_{E_{jj}}^n(\zeta)\,d\sigma(\zeta)
$$
and therefore
\begin{align*}
\left\|
(d/d\xi_1)^d \Phi_{\xi,j}(z)
\right\|_{HS}
&\leq (n+1)\,|z|^{2d}\, \sup_{x\geq 0}|(d/dx)^d\cos\sqrt{x}|\, \left\|\int_{S^3}Q_{E_{jj}}^n(\zeta)\,d\sigma(\zeta)\right\|_{HS}
\\
&\leq C_d\, (n+1)\,|z|^{2d}.
\end{align*}
The conclusion follows taking $N_d$ sufficiently large and by Lemma~\ref{norme-f-Af}.
\end{proof}

Assume now that  $g=\cG_{\tau_{m,n}}f$ admits a smooth extension $u$ on $\bR^2$ with Taylor series in~$(0,0)$
\[
%\label{jet}
\sum_{p,q}\frac{a_{p,q}}{p!q!}\xi_1^p\xi_2^q\ .
\]

If
$$
c_{d,j}=(d/d\xi_1)^d_{|_{\xi_1=0}}\cG_{\tau_{m,n}}f\big(\xi_1,(-n+2j)\xi_1\big)\ ,
$$
the following relations must hold for all $d\in\bN$ and $j=0,\dots,n$:
\[
%\label{system}
 c_{d,j}=\sum_{p+q=d}(-n+2j)^q\binom{d}{q}a_{p,q}\ .
\]

For each $d$ we obtain an $(n+1)\times (d+1)$ linear system $B_da_d=c_d$, where
\beas
a_d&=\Big(a_{d,0},\dots,{\tbinom{d}{q}}a_{d-q,q},\dots,a_{0,d}\Big)\ ,\\
 c_d&=(c_{d,0},\dots,c_{d,n})\ ,\\
 B_d&=(b_{j,q})=\big((-n+2j)^q\big)\ .
\eeas

\begin{lemma}\label{Cramer}
For every $d\in\bN$ the system $B_da_d=c_d$ admits a solution $a_d$ such that
$$
\binom{d}{q} |a_{d-q,q}| \le C_d\, (1+n)^{1+d/2}\,  \|f\|_{(N_d)}\ ,
$$
with $C_d$ independent of $n$.
\end{lemma}

\begin{proof}
We say that a set of indices $j\in\{0,\dots,n\}$ is ``central'' if it has the form $[p,n-p]$ or $[p+1,n-p]$, depending on the parity of the left-out elements.

Assume $d\ge n$. Observing that all $(n+1)\times(n+1)$ minors of consecutive columns of $B_d$ are essentially Vandermonde determinants, we have that the matrix $B_d$ has rank $n+1$ and the system is solvable, with infinite solutions if $d>n$. This case however can be reduced to the case $d=n$ by looking for a solution $a_d$ with $a_{d-q,q}=0$ for $q>n$. By Cramer's rule, 
\bea\label{sistema}
\binom{d}{q}|a_{d-q,q}|&\le\sum_{j=0}^n  |c_{d,j}|\,\left| \frac{V_{j,q}}V\right| \ ,
\eea
where $V$ is the full Vandermonde determinant with nodes  $t_j=-n+2j$
and $V_{j,q}$ are its cofactors. 
Expressing $V_{j,q}$ in terms of Schur polynomials,  cf. \cite{FH}, we have
$$
\left| \frac{V_{j,q}}V\right| =\bigg|
\sum_{k_1<k_2<\cdots<k_{n-q}\,,\,k_i\ne j}\frac{t_{k_1}t_{k_2}\cdots t_{k_{n-q}}}{\prod_{m\ne j}(t_m-t_j)}
\bigg|\le 
\sum_{k_1<k_2<\cdots<k_{n-q}\,,\,k_i\ne j}
\left|
\frac {t_{k_1}t_{k_2}\cdots t_{k_{n-q}}}{\prod_{i\ne j}(t_i-t_j)}
\right|\ .
$$

The largest numerator occurs for $\{0,\dots,n\}\setminus\{k_1,\dots,k_{n-q}\}$ central, i.e.,
\beas
\big|t_{k_0}t_{k_1}\cdots t_{k_{n-q}}\big|&\le \big|t_0t_nt_1t_{n-1}\cdots\big|\\
&=\begin{cases} n^2(n-2)^2\cdots q^2& \text{ if $n-q$ is odd}
\\
n^2(n-2)^2\cdots(q+1)^2(q-1)& \text{ if $n-q$ is even}
\end{cases}\\
&\le 2(n!!)^2\ .
\eeas

Similarly, the smallest denominator is obtained for $\{j\}$ central, so that
$$
\Big|\prod_{k\ne j}(t_k-t_j)\Big|\gtrsim 2(n!!)^2\ ,
$$
and
$$
\left|
\frac {t_{k_1}t_{k_2}\cdots t_{k_{n-q}}}{\prod_{i\ne j}(t_i-t_j)}
\right|\le1\ .
$$
Therefore
$\left|V_{j,q}/V\right|\le \binom{n}{q}\le n^{d/2}$ and, by \eqref{sistema},
$$
\binom{d}{q}|a_{d-q,q}| \le n^{1+d/2}\max_{0\le j\le n} |c_{d,j}|\ .
$$ 

\medskip
Assuming now $d< n$, Theorem \ref{(S)ridotto} guarantees that the system $B_da_d=c_d$ is solvable. Since all the maximal minors of $B_d$ are nonvanishing Vandermonde determinants, the solution is unique and we can apply Cramer's rule to the square submatrix formed by the $d+1$ central rows of $B_d$.

If $d$ and $n$ have the same parity, the system is exactly the same considered above, only with $d$ in place of $n$.
If $d$ and $n$ have different parities, the system is slightly different, but a repetition of the previous arguments leads to the same conclusion.
\end{proof}

Combining together the two lemmas \ref{Gderivatives} and \ref{Cramer}, we obtain the following asymptotic expansion 

\begin{corollary}\label{mn-jet}
Let  $f\in\cS(G)^{{\rm Int}(K)}_{\tau_{m,n}}$. For every $(m,n)\in E$ and $d\in \bN$, 
there exist coefficients $a_{d-q,q}$, $q=0,\dots, d$, and $N_d\in\bN$ such that, for all $j=0,\ldots, n$,
\begin{enumerate}[\rm(i)]
\item
$
\binom{d}{q} |a_{d-q,q}| \le C_d \, n^{1+d/2}\,  \|f\|_{(N_d)}
$
\item
$\displaystyle
(d/d\xi_1)^d_{|_{\xi_1=0}} \cG_{\tau_{m,n}}f\big(\xi_1,(-n+2j)\xi_1\big) =\sum_{q=0}^{d}
(-n+2j)^q \binom{d}{q}a_{d-q,q}$\  ; equivalently,  for $\xi=(\xi_1,\xi_2)\in\Sigma^n_\cD$, 
\be\label{jet2}
\cG_{\tau_{m,n}}f(\xi)\underset{\xi\to0}\sim {\displaystyle\sum_{d=0}^\infty\frac1{d!}\sum_{q=0}^d \binom{d}{q} a_{d-q,q}\xi_1^{d-q}\xi_2^q} \ .
\ee
\end{enumerate}
\end{corollary}

\bigskip
\subsection{Jets and smooth extensions on the full spectrum}
\quad
\medskip

We first consider a single $K$-type and construct  smooth functions on $\bR^2$, supported on the unit disk and with Taylor development \eqref{jet2} at 0.

The standard way to do so, cf. \cite[Theorem~1.2.6]{Hor},
consists in defining
\be\label{borel}
h(\xi)=\sum_{d\in\bN}\ph(\xi/\eps_d)\frac1{d!}\sum_{q=0}^d \binom{d}{q} a_{d-q,q}\xi_1^{d-q}\xi_2^q=\sum_{d\in\bN}h_d(\xi)\ ,
\ee
where $\ph\in C^\infty_c$ is supported for $|\xi|\le1$ and is equal to 1 for $|\xi|\le1/2$ and the coefficients $\eps_d\in(0,1]$ are so chosen that the series in the right-hand side converges normally in every $C^N$-norm.  

We follow this procedure keeping track at the same time of the norm estimates and of their dependence on the parameters $m,n$.

\begin{lemma}\label{jet-extension}
Let  $f\in \cS(G)^{{\rm Int}(K)}_{\tau_{m,n}}$ with Taylor development \eqref{jet2} and $M\in\bN$. There exists a function $h=h_{m,n,M}\in C^\infty_c(\bR^2)$  as in~\eqref{borel}
supported in the unit disc and such that, for every $k\leq M$,
\be\label{hCk}
\|h\|_{C^k}\le A_M\,n^{1+M/2}\|f\|_{(N_M)}+ r_{m,n}\ ,
\ee
where   $A_M>0$ is independent of $m,n$ and 
 $r_{m,n}$ is independent of $f$ and rapidly decaying in~$(m,n)$.

\end{lemma}
 \begin{proof}
In~\eqref{borel} let  $\psi_{d,q}(\xi)=\ph(\xi)\xi_1^{d-q}\xi_2^q$, so that
$\displaystyle
h_d(\xi)= \frac1{d!}\sum_{q=0}^d \binom{d}{q} a_{d-q,q}\,\eps_d^d\,\psi_{d,q}(\xi/\eps_d)$.

By Corollary \ref{mn-jet}, for every $k\in\bN$, 
$$
\|h_d\|_{C^k}\le \frac{C_d}{d!}n^{1+d/2}\|f\|_{(N_d)}\eps_d^{d-k}\sum_{q=0}^d\|\psi_{d,q}\|_{C^k}\ .
$$

With $\al_d=C_d\sum_{q=0}^d\|\psi_{d,q}\|_{C^{d-1}}$, we choose
$$
\eps_{d,m,n,M}=\begin{cases}1&\text{ if }d\leq M\\
\frac1{(n+|m|)!(1+\al_d\|f\|_{(N_d)})}&\text{ if }d> M\ . \end{cases}
$$

Then, 
$$
\sum_{d> M}\|h_d\|_{C^{d-1}}
\le \sum_{d\ge M}\frac{n^{1+d/2}}{d!(n+|m|)!}
\le\frac{ne^{\sqrt n}}{(n+|m|)!}\overset{\rm def}=r_{m,n}\ .
$$ 
This implies that the series $\sum_{d\in\bN}\|h_d \|_{C^k}$ converges for every $k$, so that $h\in C^\infty$.

Notice that  in Lemma~\ref{Gderivatives} the sequence $\{N_d\}$ can be    choosen to be  increasing. 
Then, for $k\leq  M$, 
$$
\sum_{d\leq M}\|h_d\|_{C^k}
\le \sum_{d\leq M}\|h_d\|_{C^{M}}
\le \sum_{d\le M} n^{1+d/2}\,\|f\|_{(N_d)}\frac{C_d}{d!}\,\sum_{q=0}^d\|\psi_{d,q}\|_{C^M}
\le A_M\, n^{1+M/2}\,\|f\|_{(N_M)} ,
$$
where 
$\displaystyle A_M=\sum_{d\le M} \frac{C_d}{d!}\,\sum_{q=0}^d\|\psi_{d,q}\|_{C^M}$.
This implies \eqref{hCk}. Rapid decay of $r_{m,n}$ is trivial.
\end{proof}

\subsection{Extension of spherical transforms rapidly vanishing at $0$}
\quad
\medskip

\begin{proposition}\label{p:nulla}
Suppose that $u$  in $\cS(G)^{{\rm Int}(K)}_{\tau_{m,n}}$ is such that 
$$
\left(\frac{d}{d\xi_1}\right)^q_{|_{\xi_1=0}} \cG_{\tau_{m,n}} u(\xi_1,\xi_1(-n+2j))=0\qquad \forall j=0,\dots, n,\quad\forall q\geq 0.
$$
Then there exists $v_{m,n}$ in $\cS(\bR^2)$  
 such that 
$$
v_{m,n}(\xi_1,\xi_1(-n+2j))=  \cG_{\tau_{m,n}}u (\xi_1,\xi_1(-n+2j))\qquad  \forall \xi_1\geq 0 
$$
and for every $N\geq 0$ there exist constants $C_N, N'$ depending only on $N$ such that 
$$
\|v_{m,n}\|_{(N)}
 \leq C_N\, 
\|u_{\tau_{m,n}}\|_{(N')} .
$$
\end{proposition}

\begin{proof} 
Let   $\eta$ be a bump function in $ C^\infty_c(\bR)$ supported in  $\big[-\frac12, \frac12\big]$ and equal to $1$ in a neighbourhood of the origin. Define the function $v=v_{m,n}$ on $\bR^2$ by the rule
$$
v(\xi_1,\xi_2)=
\begin{cases}
\displaystyle \sum_{j=0}^n  \cG_{\tau_{m,n}}u(\xi_1,\xi_1(-n+2j))\, \eta\left(\tfrac{\xi_2-\xi_1(-n+2j)}{\xi_1}\right)&\xi_1>0
\\
0&\xi_1\leq 0.
\end{cases}
$$
It is straightforward to show that $v$ extends $ \cG_{\tau_{m,n}}u$ to $\bR^2$. We now check the required norm estimates.

For every $j=0,1,\cdots, n$, define 
$$\eta_j(x)=\eta\left(x-(-n+2j)\right), \qquad \forall x\in \bR$$ 
$$
g_j(\xi_1)=g_{j,m,n}(\xi_1)= \cG_{\tau_{m,n}}u (\xi_1,\xi_1(-n+2j)), \qquad \forall \xi_1\geq 0
$$ 
and note that for every $\xi_1>0$
\beas
\partial^p_{\xi_1}\partial^q_{\xi_2}v(\xi_1,\xi_2)
&=  \sum_{j=0}^n \partial^p_{\xi_1}\left(
g_j(\xi_1)\,\xi_1^{-q}\, \eta_j^{(q)}  \left(\tfrac{\xi_2}{\xi_1}\right) \right)
\\
&=
\sum_{j=0}^n\sum_{s=0}^p\binom{p}{s}  
g^{(p-s)}_j(\xi_1)\,\partial^s_{\xi_1}\left(\xi_1^{-q}\, \eta_j^{(q)}  \left(\tfrac{\xi_2}{\xi_1}\right)\right).
\eeas
Moreover by induction one can check that for some coefficients $c_{r,s}$ depending only on $p,q$
$$
\partial^s_{\xi_1}\left(\xi_1^{-q}\, \eta_j^{(q)}  \left(\tfrac{\xi_2}{\xi_1}\right)\right)
=\sum_{r=0}^s c_{r,s}\,\xi_1^{-s-q}\,\left(\tfrac{\xi_2}{\xi_1}\right) ^r \eta_j^{(q+r)} \left(\tfrac{\xi_2}{\xi_1}\right),
$$
so that 
$$
\partial^p_{\xi_1}\partial^q_{\xi_2}v(\xi_1,\xi_2)
 =
\sum_{j=0}^n\sum_{s=0}^p  \sum_{r=0}^s \binom{p}{s}c_{r,s}\,\xi_1^{-s-q}\,g^{(p-s)}_j(\xi_1)\,\left(\tfrac{\xi_2}{\xi_1}\right) ^r \eta_j^{(q+r)} \left(\tfrac{\xi_2}{\xi_1}\right) .
$$

Since $\cG_{\tau_{m,n}}u$ vanishes rapidly at the origin,
 for any integer $q\geq 0$ there exists $\theta_q\in (0,1)$ such that
 for any $\xi_1\geq 0$
\[
%\label{Lagrange}
\xi_1^{-q}\, g_{j,m,n}^{(p)}(\xi_1)=\frac1{q!}\,
 g_{j,m,n}^{(p+q)}(\theta_q\xi_1) .
\]
Since for $j=0,1,\ldots$ the function $t\longmapsto t^r\, \eta_j^{(q+r)} (t)$ is still a bump function, in view of  Lemma~\ref{Gderivatives},  $v\in C^\infty(\bR^2)$ and the required norm estimates follow.
\end{proof}

We can now conclude that property (S') in Theorem \ref{rendiconti} is verified.

\begin{proposition}\label{S'}
Let  $f$ be in $\cS(G)^{{\rm Int}(K)}$ and $N$ in $\bN$. Then, for every $(m,n)\in E$, $\cG_{\tau_{m,n}} f_{m,n} $ admits a Schwartz extension $u_{m,n}^N$ from $\Sigma_\cD^n$ to $\bR^2$ such that $\|u_{m,n}^N\|_{(N)}$ is rapidly decaying in $(m,n)$.
\end{proposition}

\begin{proof}
For $M$ to be chosen afterwards, let $h_{m,n,M}$ be the function defined in Lemma~\ref{jet-extension}. Since any Schwartz norm of $f_{m,n}$ is rapidly decaying in $(m,n)$, the $M$-Schwartz norm of $h_{m,n,M}$ is also rapidly decaying in $(m,n)$. 
Moreover, let $g_{m,n}=g_{m,n,M}=\cG_{\tau_{m,n}}^{-1}\big({h_{m,n,M}}_{|_{\Sigma^n_\cD}}\big)$. 
Then $g_{m,n}$ is a Schwartz function on $G$ of type $\tau_{m,n}$ and
\[
\left(\frac{d}{d\xi_1}\right)^q_{|_{\xi_1=0}} \cG_{\tau_{m,n}}(f_{m,n}-g_{m,n})(\xi_1,\xi_1(-2j+n))=0
\qquad\forall j=0,\dots, n,\quad\forall q\geq 0.
\]
By Proposition~\ref{p:nulla}, there exists $v_{m,n}=v_{m,n,M}$ in $\cS(\bR^2)$  
 such that 
$$
v_{m,n}(\xi_1,\xi_1(-n+2j))=  \cG_{\tau_{m,n}}(f_{m,n}-g_{m,n}) (\xi_1,\xi_1(-n+2j))\qquad  \forall \xi_1\geq 0,\quad j=0,\ldots,n.
$$

Applying Corollary~\ref{7.3+7.6} (ii),  we obtain
\begin{align*}
\|v_{m,n,M}\|_{(N)}&\leq C_N\, \|f_{m,n}-g_{m,n,M}\|_{(N')}
\\
&\leq C_N\, \|f_{m,n}\|_{(N')}+\|\cG_{\tau_{m,n}}^{-1}h_{m,n,M}\|_{(N')}
\\
&\leq C_N\, \|f_{m,n}\|_{(N')}+C_{m,n,N}\|h_{m,n,M}\|_{(N'')}
\end{align*}
where $N',N''$ depend only on $N$ and the constant $C_{m,n,N}$ has polynomial growth in $(m,n)$.
Choosing $M$ bigger than $N$ and $N''$ and letting
\[
u^N_{m,n}=v_{m,n,M}+h_{m,n,M}
\]
we obtain a Schwartz extension to $\bR^2$ whose $N$-Schwartz norm is rapidly decaying in $(m,n)$. 
\end{proof}

\begin{theorem}\label{correspondence}
The Gelfand transform is a topological   isomorphism of 
 $\cS(G)^{{\rm Int}(K)}$  onto $\cS(\Sigma_\cD)$.
\end{theorem}

\begin{proof} 
Let $f$ be in $\cS(G)^{{\rm Int}(K)}$. By a diagonal process, cf. \cite{ADR3}, we derive from Proposition~\ref{S'} the existence of a single sequence $\{u_{m,n}\}$ of Schwartz functions on $\bR^2$, each extending the corresponding transform $\cG_{\tau_{m,n}}f_{m,n}$ and with Schwartz norms of any order rapidly decaying in $(m,n)$. It is then obvious that the function 
$$
v(\xi_1,\xi_2,n^2+2n,m)=u_{m,n}(\xi_1,\xi_2)\ ,\qquad (\xi_1,\xi_2)\in\bR^2\ ,
$$
extends to a Schwartz function on $\bR^4$ which coincides with $\cG f$ on $\Sigma_\cD$. This implies that $\cG$ maps $\cS(G)^{{\rm Int}(K)}$ into $\cS(\Sigma_\cD)$.

It follows from Corollary \ref{7.3+7.6} that $\cG$ is surjective and  $\cG^{-1}$ is continuous. We can then  apply the open mapping theorem again to conclude that $\cG$ is an isomorphism.
\end{proof}

For completeness, we derive from Theorem \ref{correspondence} one last property of the spherical transforms $\cGn$  that has not been discussed so far.

In Proposition \ref{inversocontinuo}  we showed that the inverses of these transforms satisfy Schwartz norm estimates with pairs $(M,N_M)$ of orders that do not depend on $(m,n)$ and with constants that grow polynomially in $(m,n)$, but norm estimates for the transforms themselves could not be established with the tools developed there. We do it now, in the following form. 

\begin{corollary}\label{Gnpolinomiale} 
For every $M\in\bN$ there exist $N_M,Q_M\in\bN$ and $A_M>0$ such that, for every $n\in\bN$ and $F\in\cS\big(\bC^2,{\rm End}(V_n)\big)^K$, the spherical transform $\cG_nF$ extends to $g_M\in\cS(\bR^2)$ satisfying 
\be\label{EndVn-estimates}
\|g_M\|_{(M)}\le A_{M}(n+1)^{Q_M}\|F\|_{(N_M)}\ .
\ee
\end{corollary}

\begin{proof}
It is convenient  to replace the $L^2$-Schwartz norms in \eqref{M-norms} with the following equivalent $L^\infty$-Schwartz norms:
$$
\|g\|_{(M,\bR^d)}=\max_{0\leq p\leq M}\sup_{\xi\in\bR^d}(1+|\xi|^2)^M\big|\Delta_\xi^p g(\xi)\big| \ ,
$$ 
on the spherical transform side (with $d=2$ or 4), 
$$
\|f\|_{(M,G)}=\max_{0\leq p,q\leq M}\sup_{(k,z)\in G}(1+|z|^2)^M\big|\Delta_z^p(\Cas-X_4^2)^q f(k,z)\big| \ ,
$$ 
for functions  on $G$, and
$$
\|F\|_{(M,\bC^2)}=\max_{0\leq p\leq N_M}\sup_{(k,z)\in G}(1+|z|^2)^{N_M}\big\|\Delta_z^p F(z)\big\|_{HS}\ ,
$$
for ${\rm End}(V_n)$-valued functions on $\bC^2$.

The equivalence mentioned above is a well-known fact.  It is also easy to prove that, in the case of ${\rm End}(V_n)$-valued functions, the constants involved in the equivalence have polynomial growth in $n$.
Hence it suffices to establish \eqref{EndVn-estimates} with the $(M,\bR^d)$, resp. $(N_M,\bC^2)$ norms.

Given $F\in\cS\big(\bC^2,{\rm End}(V_n)\big)^K$, let $f=A_{0,n}\inv F\in\cS(G)^K_{\tau_{0,n}}$.
By Theorem \ref{correspondence}, given $M\in\bN$, its (strong) spherical transform $\cG f$ defined on $\Sigma_\cD$ admits a Schwartz extension $h_M$ on $\bR^4$ such that
$$
\|h_M\|_{(2M,\bR^4)}\le C_M\|f\|_{(N_M,G)}\ ,
$$
with $N_M$ and $C_M$ independent of $n$. 

Decompose $\xi\in\bR^4$ as $(\xi',\xi'')$ with $\xi'=(\xi_1,\xi_2)$, $\xi''=(\xi_3,\xi_4)$ and define $\xi''_n=\big(n(n+2),0\big)$.

Then $g_M (\xi')=h_M(\xi',\xi''_n)$ extends $\cG_{\tau_{0,n}}f=\cG_nF$. Using the inequality
$$
(n+1)^2\big(1+|\xi'|^2\big)\le  \big(1+|\xi'|^2+|\xi''_n|^2\big)^2\ ,
$$
we obtain that
$$
\|g_M\|_{(M,\bR^2)}\le (n+1)^{-2M}\|h_M\|_{(2M,\bR^4)}\ .
$$

On the other hand, by Remark \ref{rem:fAmnf}, 
$$
f(k,z)=(n+1)\sum_{i,j}f_{ij}(z)\overline{\big(\tau_{0,n}(k)\big)_{ij}} ,
$$
where $\big(f_{ij}(z)\big)_{i,j}=F(z)$. It follows that
$(\Cas-X_0^2) f=n(n+2) f$, so that
\beas
\|f\|_{(N_M,G)}&\le(n+1)^{2N_M} \max_{0\leq p\leq N_M}\sup_{(k,z)\in G}(1+|z|^2)^{N_M}\big|\Delta_z^p f(k,z)\big|\\
&\le (n+1)^{2N_M}\|F\|_{(N_M,\bC^2)}\ .
\eeas

Therefore
\[
\|g_M\|_{(M,\bR^2)}\le (n+1)^{2N_M-2M}\|F\|_{(N_M,\bC^2)}\ .\qedhere
\]
\end{proof}

\end{document}